\documentclass[11pt,leqno]{amsart}
\usepackage{amscd}
\usepackage{color}
\usepackage[symbol]{footmisc}
\usepackage{amssymb}
\usepackage{amsfonts}
\usepackage{latexsym}
\usepackage{verbatim}
\usepackage{bbm}
\usepackage[shortlabels]{enumitem}
 \usepackage[backref=page]{hyperref}
 \renewcommand*{\backref}[1]{}
\renewcommand*{\backrefalt}[4]{%
 	\ifcase #1 (Not cited).%
	\or        (Cited on page~#2).%
	\else      (Cited on pages~#2).%
	\fi}
\hypersetup{colorlinks=true,linkcolor=blue,citecolor=blue,urlcolor=black}

\newcommand{\R}{\mathbb{R}}
\newcommand{\C}{\mathbb{C}}
\renewcommand{\H}{\mathbb{H}}
\newcommand{\pf}{\mathrm{pf}}
\newcommand{\Ric}{\mathrm{Ric}}
\renewcommand{\epsilon}{\varepsilon}
\renewcommand{\phi}{\varphi}
\renewcommand{\Re}{\mathrm{Re}}

\theoremstyle{plain}
\newtheorem{thm}{Theorem}[section]
\newtheorem{prop}[thm]{Proposition}
\newtheorem{lem}[thm]{Lemma}
\newtheorem{cor}[thm]{Corollary}

\newtheorem{quest}[thm]{Question}

\theoremstyle{definition}
\newtheorem{defn}[thm]{Definition}
\newtheorem{rmk}[thm]{Remark}
\newtheorem{es}[thm]{Example}
\title{Special metrics in hypercomplex Geometry}

\begin{document}
 
\thanks{This work was supported by GNSAGA of INdAM}

\address{(Elia Fusi) Dipartimento di Matematica G. Peano, Universit\`a di Torino, Via Carlo Alberto 10, 10123, Torino, Italy.}
\email{elia.fusi@unito.it}

\address{(Giovanni Gentili) Dipartimento di Matematica G. Peano, Universit\`a di Torino, Via Carlo Alberto 10, 10123, Torino, Italy.}
\email{giovanni.gentili@unito.it}

\author{Elia Fusi and Giovanni Gentili}


\begin{abstract}
We investigate the existence and geometric properties of special hyperhermitian metrics.
First of all, we characterise hypercomplex structures with Obata holonomy in  ${\rm SL}(n, \H)$ in terms of the existence of \emph{quaternionic Gauduchon metrics} together with the vanishing of a hypercomplex cohomological invariant. In view of this, the quaternionic Gauduchon and quaternionic balanced conditions are investigated at length: we describe their properties and determine criteria for their existence. Furthermore, we prove an incompatibility result concerning \emph{strong HKT}  and  \emph{balanced hyperhermitian metrics}, confirming an open conjecture by Fino and Vezzoni in the hypercomplex framework. Finally, we introduce an Einstein-type condition,  determining basic properties,  obstructions and providing examples.  In particular, we show that Joyce's manifolds always admit such type of metrics.
\end{abstract}

\maketitle

\section{Introduction}
  
The investigation of special metrics has become a very important research field in   complex non-K\"ahler Geometry, where there is a spreading interest in the study of Hermitian metrics satisfying weaker conditions than the K\"ahler one. In Hypercomplex Geometry, the metric notion that is most closely believed to be the analogue of the K\"ahler condition is that of a \emph{hyperk\"ahler  with torsion metric}, in short \emph{HKT}, introduced by Howe and Papadopoulos \cite{Howe-Papadopoulos (2000)}. The birth of HKT metrics comes from the study of supersymmetry with Wess-Zumino term, black holes and de Sitter supergravity, see \cite{Grover-Gutowski-Herdeiro-Sabra (2009),Gutowski-Papadopoulos,Gutowski-Sabra (2011)} and references therein. Very soon, starting from  \cite{Grantcharov-Poon (2000),Verbitsky (2002)}, HKT metrics gained particular interest also from the mathematical perspective and many properties are now known \cite{Alesker-Verbitsky (2006),BDV,Barberis-Fino,DF1,DF2,Fino-Grantcharov,IM,Ivanov-Petkov,Verbitsky (2009)}. The mathematical appeal of HKT manifolds has also been renewed by the, yet to be proven, quaternionic Calabi conjecture formulated by Alesker and Verbitsky \cite{Alesker-Verbitsky (2010)}.

\smallskip
Recall that a hypercomplex manifold $ (M^n,\mathsf{H}) $ is the data of a $ 4n $-dimensional smooth manifold $ M $ and a $ \mathrm{GL}(n,\H) $-structure $ \mathsf{H} $, where $ \H $ is the skew field of quaternions. 
A Riemannian metric $ g $ on $M$ which is Hermitian with respect to every $ L\in \mathsf{H} $ is called \emph{hyperhermitian}. Fixing  a pair of anticommuting  complex structures $ I, J\in \mathsf{H} $, the metric $ g $ can be equivalently described by means of a form $ \Omega $ of type $(2,0)$ with respect to $ I $, which is positive in a suitable sense. The metric $g$ is HKT if the Bismut connections of $ (g,L) $ all coincide, for $ L\in \mathsf{H} $. Thanks to \cite{Grantcharov-Poon (2000)}, this is equivalent to
\begin{equation}\label{eq:HKT}
\partial \Omega=0\,,
\end{equation}
where $ d=\partial+\bar \partial $ is the splitting of the exterior differential induced by $ I $. 

The fact that all complex structures in $ \mathsf{H} $ are integrable is equivalent to the existence of a unique torsion-free connection $ \nabla^{\mathrm{Ob}} $ such that $ \nabla^{\mathrm{Ob}}L=0 $, for all $ L\in \mathsf{H} $, called the \emph{Obata connection} \cite{Obata (1956)}. 
By definition,  ${\rm Hol}(\nabla^{{\rm Ob}})\subseteq \mathrm{GL}(n,\H) $. Whenever the holonomy lies in the commutator subgroup $ \mathrm{SL}(n,\H) $,  $ (M,\mathsf{H}) $ is called in the literature an {\em $ \mathrm{SL}(n,\H) $-manifold}, see \cite{GLV,Ivanov-Petkov,Lejmi-Weber,Lejmi-Tardini,Verbitsky (2007),Verbitsky (2010)}. The $ \mathrm{SL}(n,\H) $ condition implies that the canonical bundle $ K(M,L):=\Lambda^{2n,0}_LM $ is holomorphically trivial for all $L\in \mathsf{H}$, see \cite{Verbitsky (2007)}. Verbitsky proved that if on a compact HKT manifold there exists an $L\in \mathsf{H}$ such that $K(M,L)$ is holomorphically trivial, then the manifold is $\mathrm{SL}(n,\H)$, see \cite[Theorem 2.3]{Verbitsky (2007)}. On the other hand, Andrada and Tolcachier recently showed that dropping the HKT assumption Verbitsky's result fails \cite[Example 6.3]{AT}. Hence, it is natural to seek weaker metric conditions for which the equivalence obtained in \cite[Theorem 2.3]{Verbitsky (2007)} still holds. In this direction, a natural candidate is provided by \cite[Proposition 16]{GLV}, where it is shown that on compact ${\rm SL}(n, \H)$-manifolds there always exists a \emph{quaternionic Gauduchon metric}, i.e. a hyperhermitian metric satisfying
\begin{equation}\label{eq:qPaul}
\partial \partial_J \Omega^{n-1}=0\,,
\end{equation}
where $ \partial_J:=J^{-1}\bar \partial J $. Note that condition \eqref{eq:qPaul} does not depend on the choice of anticommuting $I,J\in \mathsf{H}$, see Section \ref{Sec:THM1.1}. Furthermore, the $\mathrm{SL}(n,\H)$ condition and the existence of special hyperhermitian metrics are deeply related to an invariant of the hypercomplex structure which we call the \emph{first quaternionic Bott-Chern class}. This class, denoted $c_1^{\mathrm{qBC}}(M,I,J)$, is defined as the quaternionic Bott-Chern cohomology class $[\partial_J \alpha_\Omega]_{\mathrm{qBC}}$, where $\alpha_\Omega$ is the $1$-form determined by the relation $\partial \bar \Omega^n=\alpha_\Omega \wedge \bar \Omega^n$. Such a class does not depend on the choice of the hyperhermitian metric. Furthermore, even if the definition of  $c_1^{\mathrm{qBC}}(M,I,J)$ depends on $I$ and $J$, the condition of $c_1^{\mathrm{qBC}}(M,I,J)$ being positive, zero, or negative does not, see Proposition \ref{garullabarulla}. For this reason, we can simply write $c_1^{\mathrm{qBC}}(M,\mathsf{H})>0$, $c_1^{\mathrm{qBC}}(M,\mathsf{H})=0$, and $c_1^{\mathrm{qBC}}(M,\mathsf{H})<0$, accordingly. 

We can then state our first main theorem:


\begin{thm}\label{Thm:main1.2}
A compact hypercomplex manifold $(M^n,\mathsf{H})$ is an $\mathrm{SL}(n, \H)$-manifold if and only if
admits a quaternionic Gauduchon metric and its first quaternionic Bott-Chern class vanishes.
\end{thm}

\noindent Note that the theorem implies that a compact hypercomplex manifold $(M^n,\mathsf{H})$ admitting a quaternionic Gauduchon metric is $ \mathrm{SL}(n,\H) $ if and only if there exists $L\in \mathsf{H}$ such that $K(M,L)$ is holomorphically trivial, because the existence of such an $L$ implies $c_{1}^{{\rm qBC}}(M,{\sf H})=0$, see Remark \ref{rmk:implications}.

\noindent Under the HKT assumption, it is known that the $ \mathrm{SL}(n,\H) $ condition implies the validity of the $ \partial \partial_J $-Lemma \cite[Theorem 6]{GLV}. As a consequence of Theorem \ref{Thm:main1.2} we achieve the equivalence of these two properties.

\begin{cor}\label{Cor:main1.3}
Let $ (M^n,\mathsf{H}) $ be a compact hypercomplex manifold admitting a compatible HKT metric. Then, the $ \mathrm{SL}(n,\H) $ condition and the $ \partial \partial_J $-Lemma are equivalent.
\end{cor}

Theorem \ref{Thm:main1.2} motivates the study of quaternionic Gauduchon metrics. Recall that a Hermitian metric $g$ on a complex manifold $(M^m,I)$ is called {\em Gauduchon} if its fundamental form satisfies $\partial \bar \partial \omega^{m-1}=0$.
On compact hypercomplex manifolds, in view of \cite{cPaul} and Lemma \ref{lem:gaudu},  in the  conformal class of a hyperhermitian metric  there always exists a metric $g_{{\rm G}}$ which is Gauduchon with respect to all $L\in \mathsf{H}$. An analogous  result does not hold for  quaternionic Gauduchon metrics \eqref{eq:qPaul}, as a hypercomplex manifold may not admit any such metric, see  Example \ref{Ex:Andrada-Tolcachier}. In the next result we characterize the existence of such metrics in a fixed hyperhermitian conformal class $\{\Omega\}$. As a by-product, we  obtain  an existence criterion for \emph{quaternionic balanced metrics}, i.e. those satisfying $\partial \Omega^{n-1}=0$. 

\begin{thm}\label{Thm:main1.1}
Let $(M^n, \mathsf{H},g)$ be a compact hyperhermitian manifold. Then, $g$ is conformal to a unique quaternionic Gauduchon metric of unit volume if and only if $g_{{\rm G}}$ satisfies 
\begin{equation}\label{eqn:neccond1.2}
\Gamma^{{\rm Bis}}(\{\Omega_{\mathrm{G}}\}):=\int_Ms^{{\rm Bis}}(\Omega_{\mathrm{G}})\frac{\Omega_{\mathrm{G}}^n\wedge \bar \Omega_{\mathrm{G}}^{n}}{(n!)^2}\le 0\,
\end{equation}
and 
\begin{equation}\label{eqn:neccond1.1}
s^{\mathrm{Ch}}(\Omega_{\mathrm{G}})- 2\lvert \alpha_{\Omega_{\mathrm{G}}}\rvert^2_{g_{\mathrm{G}}} \text{ either vanishes identically or it changes sign.}
\end{equation}

Moreover, $g$ is conformal to  a unique quaternionic balanced metric of unit volume if and only if  $g_{{\rm G}}$ satisfies     $\Gamma^{{\rm Bis}}(\{\Omega_{\mathrm{G}}\})=0$ and \eqref{eqn:neccond1.1}.
\end{thm}

\noindent Here, fixed anticommuting $I,J\in \mathsf{H}$, we denote $\Omega_{\rm G}\in \Lambda^{2,0}_IM$ the form associated to $g_{\rm G}$,   while $ s^{\mathrm{Ch}}(\Omega_{\mathrm{G}}) $ and $ s^{{\rm Bis}}(\Omega_{\mathrm{G}}) $ are the Chern and Bismut scalar curvatures of the Gauduchon metric, respectively, which turn out to be independent from the choice of $ I $ in $ \mathsf{H} $, see Proposition \ref{Prop:scalar}. We recall that the quantity $\Gamma^{{\rm Bis}}(\{\Omega_{{\rm G}}\})$, introduced in  \cite{Barbarone}, 
is called the \emph{Gauduchon-Bismut degree} of the conformal class of $g_{{\rm G}}$.\\
We will also characterise the existence of quaternionic Gauduchon and quaternionic balanced metrics in terms of currents, see Proposition \ref{qPaulcurrent}. As we shall see, the existence of these types of special metrics imposes some restrictions on the manifold regarding the Chern and Bismut scalar curvatures as well as the Kodaira dimension.

\smallskip

The vanishing of the first quaternionic Bott-Chern class also gives a natural obstruction to the existence of {\em strong} HKT metrics on hypercomplex manifolds, see Proposition \ref{Prop:FV}.
An HKT metric is called \emph{strong} HKT if its Bismut torsion is closed, see \cite{Barberis-Fino,BFG,BFGV,Grantcharov-Poon (2000),Ivanov-Papadopoulos,Opfermann-Papadopoulos}.
Such metrics are quite rare and they impose several restrictions on the hypercomplex structure. For instance, Verbitsky showed in \cite[Theorem 5.11]{Verbitsky (2009)} that in the compact case, whenever $ n\geq 3 $, a $ (2,0) $-quaternionic Bott-Chern cohomology class cannot contain both a strong HKT and a balanced HKT metric, unless the manifold is actually hyperk\"ahler. Here we shall  improve this incompatibility result by proving the following.

\begin{thm}\label{Thm:main1.4}
Let $(M,\mathsf{H})$ be a compact hypercomplex manifold with $c_1^{\mathrm{qBC}}(M,{\sf H})\geq 0$ and $c_1^{\mathrm{qBC}}(M,{\sf H})\neq 0$. Then, $(M,\mathsf{H})$ does not admit any balanced hyperhermitian metric. In particular, if $(M,\mathsf{H})$ admits a strong HKT metric and a balanced hyperhermitian metric, then it is hyperk\"ahler.
\end{thm}

\noindent This theorem belongs to the collection of evidences towards the validity of the Fino-Vezzoni conjecture \cite[Problem 3]{FV}, which has attracted a great deal of attention over the years and several partial confirmations are now available, e.g. \cite{CZ,Ch,FGV,FP,FV,FV2,FS,GP}.

\smallskip
A strong HKT metric cannot be Chern-Ricci flat unless it is hyperk\"ahler and it is quite natural to understand which symmetries can be imposed on the Chern curvature. In this direction, we introduce an adaptation of the Einstein condition in the hyperhermitian setting by defining a  \emph{Chern-Einstein} metric as a hyperhermitian metric $g$ satisfying   
\[
\frac{\Ric^{\mathrm{Ch}}_{\omega_I}-J\Ric^{\mathrm{Ch}}_{\omega_I}}{2}=\lambda \omega_I\,, \quad \lambda \in C^\infty(M,\R)\,,
\]
where $ \omega_I=g(I\cdot,\cdot) $ and $ \Ric^{\mathrm{Ch}}_{\omega_I} $ is the corresponding Chern-Ricci form. In the compact HKT case, $ \lambda $ is necessarily a non-negative constant, in contrast with the classical K\"ahler-Einstein case. HKT-Einstein metrics with $\lambda > 0$ are such that the Lee form is a potential $1$-form in the terminology of \cite[Definition 4]{OPS}. Admitting a closed potential $1$-form is a necessary condition for having a $D(2, 1;-1) $-symmetry, which is a degeneration of the $D(2,1;\alpha)$-symmetry, studied in \cite{Michelson-Strominger}. On the other hand, when $ \lambda=0 $,  these metrics are actually balanced  HKT and, from the point of view of the quaternionic Calabi conjecture, they play the role of K\"ahler Calabi-Yau metrics on compact HKT $ \mathrm{SL}(n,\H) $-manifolds, see \cite{Verbitsky (2009)}. 
On the contrary, an HKT-Einstein metric with non vanishing Einstein constant can only live on manifolds that are not $ \mathrm{SL}(n,\H) $. 


 As it is observed in \cite{BDV},  Joyce hypercomplex manifolds are never ${\rm SL}(n, \mathbb H), $ thus providing candidates on which positive Chern-Einstein metrics may exist.
Recall that direct products of a compact semisimple Lie group $G$ with a  torus $T^k$ of suitable dimension admit left-invariant hypercomplex structures, see \cite{Joyce, SSTV}. We will call $G\times T^k $ equipped with one of these hypercomplex structures $\mathsf{H} $ a {\em Joyce hypercomplex manifold}.  Opfermann and Papadopoulos \cite{Opfermann-Papadopoulos} showed that with some choices of $\mathsf{H}$ the Cartan-Killing form can be extended to a compatible strong HKT metric (see also  \cite{Grantcharov-Poon (2000)}). However, it remains open to determine whether a general Joyce hypercomplex manifold carries HKT metrics. The next theorem settles this problem while also providing a large class of  HKT–Einstein manifolds with positive Einstein factor.

\begin{thm}\label{Thm:main1.5}
Let $ (G\times T^k,\mathsf{H}) $ be a Joyce hypercomplex manifold. Then, there exists an invariant HKT-Einstein metric with positive Einstein constant compatible with $\mathsf{H}$.
\end{thm}


\smallskip
Let us now explain in detail the organisation of the paper. In the first section, we will collect some preliminaries together with well-known results and facts in Hypercomplex Geometry. In particular, Subsection \ref{Sec:ab} will be dedicated to the study of the canonical forms $\alpha$ and $\beta$, which will be a fundamental tool throughout the paper. 
We will relate them with some known quantities such as the Lee form, the first Chern-Ricci and the Bismut-Ricci form. \\
In Section \ref{Sec:THM1.2}, we will define the first quaternionic Bott-Chern class of a hypercomplex manifold and  prove Theorem \ref{Thm:main1.2}, along with other properties of quaternionic Gauduchon and quaternionic balanced metrics.\\
In Section \ref{Sec:THM1.1}, the existence problem for quaternionic Gauduchon and quaternionic balanced metrics will be taken into account. In this section, we prove Theorem \ref{Thm:main1.1} and determine a sufficient condition for a quaternionic Gauduchon metric to exist in terms of the existence of a suitable $\bar \partial\bar \partial_J$-closed $(2n,0)$-form. \\
In Section \ref{Sec:sHKT}, we will be focusing our attention on strong HKT metrics, proving Theorem \ref{Thm:main1.4}, while, in Section \ref{Sec:Einst}, we will study the Chern-Einstein condition on hyperhermitian metrics deriving their fundamental properties. Here we also establish Theorem \ref{Thm:main1.5}. \\
The theoretical framework is enriched by some examples and a couple of constructions, collected in Section \ref{Sec:Ex}. 
The first one is an adaptation of an idea of Arroyo and Nicolini \cite[Section 5]{AN} and works on nilpotent Lie algebras, while the second was architectured by Barberis and Fino \cite{Barberis-Fino} and works with any Lie algebra. In this section, we also present some non-compact examples of HKT-Einstein manifolds with negative Einstein factor.

\medskip 

\noindent {\bf Acknowledgement.} We wish to express our gratitude to L. Vezzoni for constant encouragement and for taking interest in our paper. We also thank him together with L. Bedulli for their useful suggestions regarding the exposition. We are grateful to L. Bedulli and L. Marcocci for useful discussions regarding Theorem \ref{Thm:main1.5}. The authors wish to thank A. Fino, G. Grantcharov, M. Lejmi  and M. Sroka for their interest and their comments which have greatly improved the paper. The definition of the Einstein condition was suggested to the second-named author by M. Verbitsky, to whom he is deeply grateful.

\section{Preliminary results}\label{Sec:pre}

\subsection{Hyperhermitian structures}
Let $ (M^n,\mathsf{H}) $ be a hypercomplex manifold of real dimension $4n$, i.e. $\mathsf{H}$ is a $ 2 $-sphere of complex structures. Fixed a pair of anticommuting $I,J\in \mathsf{H}$, we have
\[
\mathsf{H}=\{ aI+bJ+cK \mid (a,b,c)\in S^2 \}\,,
\]
where $ K=IJ $. We emphasise that the role of $ (I,J) $ is not preferential and we can replace it with any pair of anticommuting complex structures in $ \mathsf{H} $. On the other hand, throughout the paper, we will always consider a fixed basis $ (I,J,K) $ for $ \mathsf{H} $. 

We will denote with $ \Lambda^{p,q}_LM $ the space of $ (p,q) $-forms with respect to $ L\in \mathsf{H} $. A Riemannian metric $g$ that is $L$-Hermitian for any $ L\in \mathsf{H} $ is called \emph{hyperhermitian}. Let $ \omega_L $ be the fundamental form of the Hermitian pair $(g,L)$. Evidently, we have that a $ I $-Hermitian metric $ g $ is hyperhermitian if and only if $ \omega_I $ is $ J $-anti-invariant. The hyperhermitian structure can also be completely described in terms of the following distinguished form:
\begin{equation}\label{eq:omeghe}
    \Omega:=\frac{\omega_{J}+\sqrt{-1}\omega_K}{2}\in \Lambda^{2,0}_IM\,.
\end{equation}
The form $ \Omega $ satisfies the following properties:
\begin{equation}\label{eq:q-pos}
J \bar \Omega= \Omega \,, \quad \Omega(X,J X)> 0\,, \quad  X\in TM\setminus \{0\}\,, \qquad \frac{\Omega^n \wedge \bar \Omega^n}{(n!)^2}=\frac{\omega^{2n}_{I}}{(2n)!}\,.
\end{equation}
Conversely, any $ \Omega \in \Lambda^{2,0}_IM $ satisfying the first two properties in \eqref{eq:q-pos} induces a hyperhermitian metric. In view of this correspondence, we shall often say that $ \Omega $ is a hyperhermitian metric, by a slight abuse of language.


Let $ d=\partial + \bar \partial  $ be the splitting of the exterior differential induced by $ I $. We will need  the following twisted operator $\partial_J:=J^{-1}\partial J$,  introduced in \cite{Verbitsky (2002)}.
Evidently,  $ \partial_J^2=0 $ and it is easy to check that $ \partial \partial_J=-\partial_J\partial$,  see e.g. \cite[Lemma 2.12]{Dinew-Sroka}. In view of these properties, highly relevant cohomology groups arise in the hypercomplex setting, see \cite{GLV}.
\begin{defn} Let $(M, \mathsf{H})$ be a hypercomplex manifold. The \emph{quaternionic Bott-Chern} and \emph{quaternionic Aeppli} cohomology groups are, respectively, defined as:
$$
H^{*, *}_{{\rm qBC}}(M)=\frac{\ker \partial\cap \ker \partial_J}{{\rm Im}\, \partial \partial_J}\,,\quad H^{*, *}_{{\rm qA}}(M):=\frac{\ker \partial\partial_J}{{\rm Im}\, \partial + {\rm Im}\,\partial_J}\,.
$$
\end{defn}

\subsection{Positive forms and currents}
Let $ (M,\mathsf{H}) $ be a hypercomplex manifold. A differential form $ \gamma \in \Lambda^{2p,2q}_IM $ is called \emph{q-real} if $ J \bar \gamma=\gamma $ and \emph{q-semipositive} (resp. \emph{q-positive}) if, additionally,
\[
\gamma(X_1,JX_1,\dots,X_{p+q},JX_{p+q})\geq 0\,\quad (\text{resp. }>0), 
\]
for any non-zero $ X_1,\dots,X_{p+q}\in TM $. In particular,  \eqref{eq:q-pos} is saying that $ \Omega $ is q-real and q-positive.\\
It will be useful to observe the following well-known fact. The proof is essentially the same as in \cite[Proof of Theorem 4.7]{Mich} adapted to the hypercomplex case, therefore we omit it.

\begin{lem}\label{Lem:bijection}
Let $ (M^n,\mathsf{H}) $ be a hypercomplex manifold, then the $(n-1)$\textsuperscript{th} wedge power is a bijective correspondence between the cone of q-positive $(2,0)$-forms and the cone of q-positive $(2n-2,0)$-forms with respect to $ I $.
\end{lem}

In what follows, we will denote with $ \mathcal{D}^{p,q}_I(M) $ the space of currents of bidegree  $ (p,q) $ with respect to $ I $. The action of $L\in \mathsf H $,   $ \partial$ and $\partial_J $ are extended to $ (p,q) $-currents by duality, see \cite[Section 6]{GLV} for details. Finally, we say that $ T\in \mathcal{D}_I^{2p,2q}(M) $ is \emph{q-real} if $ J\bar T=T $. If further $ T(\gamma)\geq 0, $ for any q-positive $ \gamma \in \Lambda^{2n-2p,2n-2q}_IM $, we say that $ T $ is \emph{q-positive}.

\subsection{Useful formulae} 
In this subsection we collect some well-known formulae. First of all, we state the following Lemma.
\begin{lem}[\cite{Alesker-Verbitsky (2006)}]\label{lemAV}
Let $ (M,\mathsf{H}) $ be a hypercomplex manifold. There is a bijective correspondence between the set of $J$-anti-invariant forms in $ \Lambda^{1,1}_IM $ and $ \Lambda^{2,0}_IM $ given by $ \gamma \mapsto \Phi(\gamma) $, where
\[
\Phi(\gamma)(X,Y)=\frac{\sqrt{-1}\gamma(JX,Y)-\gamma(KX,Y) }{2}\,, \quad X,Y\in TM\,.
\]
Furthermore, $ \Phi(\gamma) $ is q-real (resp. q-positive) if and only if $ \gamma $ is real (resp. positive).
\end{lem}

By rewriting \eqref{eq:omeghe} in terms of $ \omega_I $, we see that  $ \Omega =\Phi(\omega_I) $. Moreover, in view of Lemma \ref{lemAV},  if $ \gamma\in \Lambda^{1,1}_IM$, then, $ \Phi (\frac{\gamma-J\gamma}{2} ) $ is q-real if and only if $ \gamma-J\gamma$ is real, i.e. $ \gamma-\bar \gamma $ is $ J $-invariant. Furthermore, if $  \psi\in \Lambda_I^{1,0}M $, then 
\begin{equation}\label{Lem:ddJ}
\Phi \left(\frac{\sqrt{-1}\bar \partial \psi-\sqrt{-1}J\bar \partial\psi}{2} \right)=\frac{1}{2}\partial_J \psi\,,
\end{equation}
which can be proved by adapting \cite[Lemma 2.1]{Bedulli-Gentili-Vezzoni2} (cf. also \cite[Remark 4.1]{Sroka22}). In particular, $\partial_J \psi$ is q-real if and only if $\bar \partial \psi+ \partial \bar \psi$ is $J$-invariant.

  Under the correspondence $\Phi$, we can observe the following fact. 
If $ \gamma \in \Lambda^{1,1}_IM$, then,  by straightforward calculations,  we have 
\begin{equation}\label{eq:traces}
\begin{aligned}
\mathrm{tr}_{\Omega}\left(\Phi\left(\frac{\gamma-J\gamma}{2}\right)\right)
=\frac{1}{2}\mathrm{tr}_{\omega_I}\gamma\,,
\end{aligned}
\end{equation}
where 
$$
{\rm tr}_{\Omega}\xi:=n\frac{\xi \wedge \Omega^{n-1}}{\Omega^n}\,, \quad \xi\in \Lambda^{2,0}_IM\,.
$$
As a consequence, we recover from \eqref{Lem:ddJ}  and \eqref{eq:traces} that the operator
\[
\Delta_\Omega \colon C^\infty(M,\R) \to C^\infty(M,\R)\,, \qquad \Delta_{\Omega} \phi:=\mathrm{tr}_\Omega(\partial \partial_J \phi)=\mathrm{tr}_{\omega_I}(\sqrt{-1}\partial \bar \partial \phi)
\]
is the Chern Laplacian.

Finally, it will also be useful to recall the definition of the Hodge star operator:
\[
\psi \wedge * \zeta=g(\psi,\zeta)\frac{\Omega^n\wedge \bar \Omega^n}{(n!)^2}\,, \quad \psi,\zeta \in \Lambda^{p,q}_IM\,, 
\]
where $ g $ is the hyperhermitian metric associated to $ \Omega $. From this, we can deduce the identities:
\begin{align*}
*\Omega&=\frac{\Omega^{n-1}\wedge \bar \Omega^{n}}{n!(n-1)!}\,, \quad *\psi=-J\bar \psi \wedge\frac{\Omega^{n-1}\wedge \bar \Omega^{n}}{n!(n-1)!} \,, \quad \psi \in \Lambda^{1,0}_IM\,.
\end{align*}
Moreover, one can easily prove that:
\begin{equation}\label{eqn:starHodge}
*\zeta=-J\bar \zeta\wedge\frac{\Omega^{n-2}\wedge \bar \Omega^n}{n!(n-2)!}+{\rm tr}_{\Omega}(J \bar \zeta) \frac{\Omega^{n-1}\wedge \bar \Omega^n}{n!(n-1)!}\,, \quad \zeta \in \Lambda_I^{2,0}M\,.
\end{equation}
We conclude this subsection proving an easy lemma which will be used later in the paper.
\begin{lem}\label{Lem:Lemma2.3}
Let $(M^n, \mathsf H, g)$ be a hyperhermitian  manifold. Then, for every $\psi,\zeta\in \Lambda^{2,0}_IM$, we have:
\begin{equation}\label{eqn:hodgerel}
\psi\wedge\zeta \wedge \frac{\Omega^{n-2}}{(n-2)!}=\left({\rm tr}_{\Omega}(\psi){\rm tr}_{\Omega}(\zeta)-g(\psi, J \bar \zeta)\right) \frac{\Omega^{n}}{n!}\,.
\end{equation}
\end{lem}
\begin{proof}
Fixed $\psi,\zeta\in \Lambda^{2,0}_IM$, we have that, using \eqref{eqn:starHodge}, 
\[
\begin{aligned}
g(\psi, J \bar \zeta)\frac{\Omega^n\wedge \bar \Omega^n}{(n!)^2}=&\, -\psi\wedge \zeta\wedge\frac{\Omega^{n-2}\wedge \bar \Omega^n}{n!(n-2)!}+{\rm tr}_{\Omega}(\zeta)\psi\wedge \frac{\Omega^{n-1}\wedge \bar \Omega^n}{n!(n-1)!}\\
=&\,  -\psi\wedge \zeta\wedge\frac{\Omega^{n-2}\wedge \bar \Omega^n}{n!(n-2)!}+ {\rm tr}_{\Omega}( \zeta){\rm tr}_{\Omega}( \psi)\frac{\Omega^n\wedge \bar \Omega^n}{(n!)^2}\,.
\end{aligned}
\] Using that wedging with $\frac{\bar \Omega^n}{n!}$ is an isomorphism, we conclude. 
\end{proof}

\subsection{Canonical forms in hyperhermitian geometry}\label{Sec:ab}
Let $ (M^n,\mathsf{H},g) $ be a hyperhermitian manifold. We define the forms $ \alpha_{\Omega},\beta_{\Omega} \in \Lambda^{1,0}_IM $ via the relations:
\[
    \partial \bar \Omega^n= \alpha_{\Omega} \wedge \bar \Omega^n\,, \quad \quad \partial \Omega^{n-1}=\beta_{\Omega} \wedge \Omega^{n-1}\,.
\]
To lighten the notation,  we shall write $ \alpha $ and $ \beta $ in place of $ \alpha_{\Omega} $ and $ \beta_{\Omega} $, whenever  no confusion can be made. Let us explain why such forms are well-defined. Regarding $ \alpha $, it is well-known that $ \alpha+\bar \alpha $ is the connection $ 1 $-form $ \eta_I $ of the Obata connection on the canonical bundle $ K(M,I):=\Lambda^{2n,0}_IM $, see, for instance, \cite{Verbitsky (2007)}.
Thanks to the last identity in \eqref{eq:q-pos}, one can see that $ \eta_I=\alpha + \bar \alpha $ does not depend on $ I $. For this reason we shall drop the reference to the complex structure and simply denote it $ \eta $.\\
On the other hand, the form $ \beta $ is well-defined by the fact that the map $L^{n-1}\colon \Lambda^{1,0}_IM\to \Lambda^{2n-1,0}_IM $ such that $L^{n-1} (\gamma ):=\Omega^{n-1}\wedge \gamma $ is an isomorphism.  The proof of this property is standard.

We also mention that $ \alpha $ is $ \partial $-closed and $ \partial_J \alpha $ is q-real, see \cite[Section 10.1]{Verbitsky (2002)}. In the same fashion, one sees that 
$ \mathrm{tr}_\Omega (\partial\beta)=0 $ and $ \mathrm{tr}_\Omega(\partial_J \beta)=\mathrm{tr}_\Omega(\partial J \bar \beta ) $.

The next lemma gives alternative expressions of $ \alpha $ and $ \beta $ in terms of the adjoint $ \Lambda $ of the Lefschetz operator with respect to $\Omega$. 

\begin{lem}\label{Lem:alphabeta}
Let $ (M^n,\mathsf{H},g) $ be a hyperhermitian manifold. Then
\[
\alpha=\bar \Lambda(\partial \bar \Omega)\,, \quad \beta=\Lambda(\partial \Omega)\,.
\]
\end{lem}
\begin{proof}
Let $ Z\in T^{1,0}_IM $. Then, we have
\[
\begin{aligned}
     \iota_Z\bar \Lambda(\partial \bar \Omega)&
     =*\left(\iota_Z\partial\bar \Omega\wedge \frac{\Omega^n\wedge \bar \Omega^{n-1}}{n!(n-1)!}\right)  
     = *\left( \iota_Z\Omega^n\wedge \frac{\partial\bar \Omega\wedge \bar \Omega^{n-1}}{n!(n-1)!}\right)
     = 
     \iota_Z\alpha\,.
     \end{aligned}
    \]
Then, $\alpha=\bar \Lambda( \partial\bar \Omega)$. In a similar manner, we have
\[
\begin{aligned}
      \iota_Z\Lambda(\partial \Omega)&=*\left(\iota_Z\partial \Omega\wedge \frac{\Omega^{n-1}\wedge \bar \Omega^n}{n!(n-1)!}\right)= *\left(\iota_Z\Omega\wedge \beta\wedge\frac{\Omega^{n-1}\wedge \bar \Omega^n}{n!(n-1)!}\right)
   =\iota_Z\beta\,,
\end{aligned}
\]
concluding the proof.
\end{proof}

The interest of the forms $ \alpha $ and $ \beta $, as it turns out, is that they are strictly related to other well-known quantities.

\begin{prop}\label{Prop:alpha}
Let $(M^n,\mathsf{H},g)$ be a hyperhermitian manifold. Then,  the following hold:
\begin{enumerate}[label=(\alph*),ref=\alph*]
\item \label{alphaLee} the Lee forms of $ \omega_L $, for $ L\in \mathsf{H} $, all coincide and they are equal to
\[
\theta_g:=\alpha+ \bar \alpha+ \beta+\bar \beta\,;
\]
as a consequence, $\beta+ \bar \beta$ is independent of $L \in\mathsf H$; 
\item \label{alphaChern} the first Chern-Ricci form of $\omega_L$, for any $L\in \mathsf H $, is given by:
$$
\Ric_{\omega_L}^{\mathrm{Ch}}=dL(\alpha+\bar \alpha)=dL\eta\,;
$$
\item\label{Lem:ricdelJalpha} $$\Phi\left(\frac{\Ric^{\mathrm{Ch}}_{\omega_I}-J\Ric^{\mathrm{Ch}}_{\omega_I}}{2}\right)=\partial_J\alpha\,;$$
\item \label{alphaBismut}  the  Bismut-Ricci form of $\omega_L$, for any $L\in \mathsf H $, is given by:
\[
\Ric_{\omega_L}^{\mathrm{Bis}}=-dL(\beta+\bar \beta)\,.
\]
\end{enumerate}
\end{prop}
\begin{proof}
To prove \eqref{alphaLee},  let $ \theta_L:=-Ld^*\omega_L $ be the Lee form of $ \omega_L $. We first observe that
\[
d^*\Omega=-*\partial \left(\frac{\Omega^{n-1}\wedge\bar  \Omega^n}{n!(n-1)!}\right)=-*\left((\alpha+ \beta) \wedge \frac{ \Omega^{n-1}\wedge \bar \Omega^n}{n!(n-1)!}\right)=J(\bar \alpha+\bar \beta)\,.
\]
Hence,
\[
J\theta_J=d^*\omega_J=d^*(\Omega+\bar \Omega)=J(\bar \alpha+\bar \beta+\alpha+\beta)
\]
and similarly
\[
K\theta_K=d^*\omega_K=-\sqrt{-1}d^*(\Omega-\bar \Omega)=K(\bar \alpha+\bar \beta+\alpha+\beta)\,,
\]
therefore $ \theta_J=\theta_K. $ The same argument replacing $ J $ and $ K $ with $ K $ and $ I $ respectively, shows that $ \theta_K=\theta_I $. Thus we have $ \theta_I=\theta_J=\theta_K=:\theta_g $ and actually, for any $ L=aI+bJ+cK\in \mathsf{H} $
\[
\theta_{L}=-Ld^*\omega_{L}=-Ld^*(a\omega_I+b\omega_J+c\omega_K)=L(aI^{-1}+bJ^{-1}+cK^{-1})\theta_g=\theta_g\,,
\]
since $ aI^{-1}+bJ^{-1}+cK^{-1}=L^{-1} $. The claim about $\beta+ \bar \beta $ follows from what we just proved and the fact that $\eta$ is independent of $L\in\mathsf H$.

It is enough to show \eqref{alphaChern} and \eqref{alphaBismut} for $L=I$. Fix local $I$-holomorphic coordinates $\{z^1, \ldots, z^{2n}\}$. We set $P:=\mathrm{pf}(\Omega_{ij})$, the Pfaffian of $(\Omega_{ij})$,  and $G:=\det(g_{r\bar s})$.
Since $ \Omega^n=n!P dz^1\wedge \dots \wedge dz^{2n}$, it follows that $\bar \partial P=P\bar \alpha. $ Then, since $|P|^2=G$,  we get
\begin{equation}\label{eqnbarpartialG}
\bar \partial G=\overline{P}\bar \partial P+P\overline{\partial P}=G\bar \alpha+P\overline{\partial P}\,.
\end{equation}
Hence, using \eqref{eqnbarpartialG}, we have
\[
\begin{aligned}
\Ric_{\omega_I}^{\mathrm{Ch}}&=-\sqrt{-1}\partial\left( G^{-1}\bar \partial G \right)=
-\sqrt{-1}\left(\partial \bar \alpha+\partial(G^{-1}P\overline{\partial P})\right)
\\
&= \sqrt{-1}\left(-\partial \bar \alpha + G^{-1}P\alpha \wedge \overline{\partial P}+G^{-1}P\bar \partial( \overline{ P}\alpha )\right)
=\sqrt{-1}\left(\bar \partial \alpha-\partial \bar \alpha\right)=dI\eta.
\end{aligned}
\]
 To prove (\ref{Lem:ricdelJalpha}), we use \eqref{alphaChern} and obtain
 \[
\frac{\Ric^{\mathrm{Ch}}_{\omega_I}-J\Ric^{\mathrm{Ch}}_{\omega_I}}{2}=\frac{\sqrt{-1}}{2}\left( \bar \partial \alpha-\partial \bar \alpha-J\bar \partial \alpha+J\partial \bar \alpha \right)=\sqrt{-1}\left( \bar \partial \alpha-J\bar \partial \alpha \right)\,,
\]
where  $\bar\partial  \alpha- J \bar \partial \alpha= -\partial \bar \alpha+ J \partial \bar \alpha $, since $ \partial_J \alpha $ is q-real. Therefore, using \eqref{Lem:ddJ},  we conclude. 

Now \eqref{alphaBismut} follows from \eqref{alphaLee} and \eqref{alphaChern} together with the well-known identity
\[
\Ric_{\omega_I}^{\mathrm{Ch}}-\Ric^{\mathrm{Bis}}_{\omega_I}=dI\theta_g\,. \qedhere
\]
\end{proof}

We remark here that, when the hyperhermitian metric is HKT,  the fact that the Lee forms coincide was observed in \cite{Boyer, Ivanov-Papadopoulos}, and their relation with $\eta$ was shown in \cite[Lemma 2.2]{Bedulli-Gentili-Vezzoni2}.

We will now focus on scalar curvatures of hyperhermitian metrics. From  Proposition \ref{Prop:alpha} (\ref{Lem:ricdelJalpha}) and   \eqref{eq:traces} and we infer that
\[
s^{\mathrm{Ch}}(\omega_I):=\mathrm{tr}_{\omega_I}\Ric_{\omega_I}^{\mathrm{Ch}}=2\mathrm{tr}_{\Omega}(\partial_J \alpha)\,,
\]
where  $ s^{\mathrm{Ch}}(\omega_I)$ is the \emph{Chern scalar curvature} of $ \omega_I $. Moreover,  using Proposition \ref{Prop:alpha} (\ref{alphaBismut}) and \eqref{eq:traces}, we have that the \emph{Bismut scalar curvature} $ s^{\mathrm{Bis}}(\omega_I)$ satisfies
\begin{equation}\label{eqn:bismutscalarbeta}
s^{\mathrm{Bis}}(\omega_I):={\rm tr}_{\omega_I}{\rm Ric}^{{\rm Bis}}_{\omega_I}=-2\mathrm{tr}_{\Omega}(\partial_J \beta)\,.
\end{equation}

\begin{prop}\label{Prop:scalar}
Let $ (M^n,\mathsf{H},g) $ be a hyperhermitian manifold. Then, for all  $ L\in \mathsf{H} $,  $$
s^{\mathrm{Ch}}(\omega_I)=s^{\mathrm{Ch}}(\omega_L)\,,\quad s^{\mathrm{Bis}}(\omega_I)=s^{\mathrm{Bis}}(\omega_L)\,. 
$$
\end{prop}
\begin{proof}
We prove the proposition for Chern scalar curvatures, for Bismut scalar curvatures the argument is analogous. Observe that it is enough to prove $ s^{\mathrm{Ch}}(\omega_P)=s^{\mathrm{Ch}}(\omega_L) $, for all $ L\in \mathsf{H} $ such that $ PL=-LP $. 
 Thus,  let $ L=aJ+bK\in \mathsf{H} $ be the generic complex structure  anticommuting with $ I $. Note that
\begin{equation}\label{eq:omegaL}
\omega_L=a\omega_J+b\omega_K=a(\Omega+\bar \Omega)-\sqrt{-1}b(\Omega-\bar \Omega)=(a-\sqrt{-1}b)\Omega+(a+\sqrt{-1}b)\bar \Omega
\end{equation}
and set $ w:=a-\sqrt{-1}b $. Keeping in mind that $ |w|=1 $, we compute
\[
\begin{split}
\omega_L^{2n-1}&=\binom{2n-1}{n}\bar w\Omega^{n-1}\wedge \bar \Omega^n+\binom{2n-1}{n}w\Omega^n\wedge \bar \Omega^{n-1}\,,\quad
\omega_L^{2n}=\binom{2n}{n}\Omega^n\wedge \bar \Omega^{n}\,.
\end{split}
\]
Thus, for any $ \xi \in \Lambda^2M$, we have
\[
\mathrm{tr}_{\omega_L}\xi=n\bar w\frac{ \xi^{2,0}\wedge \Omega^{n-1}\wedge \bar \Omega^n}{\Omega^n\wedge \bar \Omega^n}+nw\frac{\xi^{0,2}\wedge \Omega^n\wedge \bar \Omega^{n-1}}{\Omega^n\wedge \bar \Omega^n}=\bar w\,\mathrm{tr}_{\Omega}(\xi^{2,0})+w\,\mathrm{tr}_{\bar \Omega}(\xi^{0,2})\,.
\]
Now, to conclude the proof we only need to compute the $ (2,0) $ and $ (0,2) $ parts of $ \Ric^{\mathrm{Ch}}_{\omega_L} $. We observe,  by expanding the Chern-Ricci forms of $ \omega_J $ and $ \omega_K $, that:
\begin{align}
\label{eq:Ric_J}
\Ric^{\mathrm{Ch}}_{\omega_J}&=J\Ric^{\mathrm{Ch}}_{\omega_J}=JdJ\eta=\partial_J\alpha+\partial_J \bar \alpha+\bar \partial_J\alpha+\overline{\partial_J \alpha}\,,\\
\label{eq:Ric_K}
\Ric^{\mathrm{Ch}}_{\omega_K}&=IJ^{-1}dJI(\alpha+\bar \alpha)=\sqrt{-1}\left(-\partial_J\alpha-\partial_J \bar \alpha+\bar \partial_J\alpha+\overline{\partial_J \alpha}\right)\,.
\end{align}
Then, since $ \Ric^{\mathrm{Ch}}_{\omega_L}=dL\eta=a\Ric^{\mathrm{Ch}}_{\omega_J}+b\Ric^{\mathrm{Ch}}_{\omega_K} $, we get
\begin{equation}\label{eq:RicL}
(\Ric^{\mathrm{Ch}}_{\omega_L})^{2,0}=w\partial_J \alpha\,, \quad (\Ric^{\mathrm{Ch}}_{\omega_L})^{0,2}=\overline{w\partial_J \alpha}\,.
\end{equation}
Finally, we conclude
\[
s^{\mathrm{Ch}}(\omega_L)=\mathrm{tr}_{\omega_L} \Ric^{\mathrm{Ch}}_{\omega_{L}}=\bar w\,\mathrm{tr}_{\Omega}\left((\Ric^{\mathrm{Ch}}_{\omega_L})^{2,0}\right)+w\,\mathrm{tr}_{\bar \Omega}\left((\Ric^{\mathrm{Ch}}_{\omega_L})^{0,2}\right)=s^{\mathrm{Ch}}(\omega_I)\,. \qedhere
\]
\end{proof}

\noindent In view of  Proposition \ref{Prop:scalar} we will omit the reference to the $ (1,1) $-form with respect to which we are considering scalar curvatures and simply denote them by $ s^{\mathrm{Ch}} $ and $ s^{\mathrm{Bis}} $. If needed, we shall write  $ s^{\mathrm{Ch}}(\Omega) $ and $ s^{\mathrm{Bis}}(\Omega) $,  specifying the hyperhermitian metric $ \Omega $ with respect to which scalar curvatures are considered.

We shall now study equivalent conditions, in terms of the hyperhermitian data, for a hyperhermitian metric to satisfy some well-known cohomological properties studied on complex manifolds. We quickly recall that a Hermitian metric $g$ on a complex manifold $(M, J)$ is said to be {\em Gauduchon} if $d^*\theta_g=0$. Furthermore, $g$ is called {\em balanced} if $\theta_g=0$. 

\noindent The following lemma identifies conditions on the form $\Omega$ in order for $\omega_I$ to be Gauduchon. 

\begin{lem}\label{lem:gaudu}
Let $ (M^n,\mathsf{H},g) $ be a  hyperhermitian manifold. Then,   the following are equivalent
\begin{enumerate}
\item\label{1gauduchon} $ \omega_I $ is Gauduchon\,;
\item\label{tuttegauduchon} $ \omega_L $ is Gauduchon for any $ L\in \mathsf{H} $\,;
\item \label{scalarigau}$ s^{\mathrm{Ch}}-s^{\mathrm{Bis}}-2|\alpha+\beta|^2=0 $\,;
\item \label{deldelJgau} $ \partial \partial_J (\Omega^{n-1}\wedge \bar \Omega^n)=0 $.
\end{enumerate}
\end{lem}
\begin{proof}
By definition of being Gauduchon, the equivalence of \eqref{1gauduchon} and \eqref{tuttegauduchon} follows from Proposition \ref{Prop:alpha} \eqref{alphaLee}.  Next,  we compute
\[
\begin{aligned}
d^*(\alpha+\beta)=*d\left(J(\bar \alpha+\bar \beta)\wedge \frac{\Omega^{n-1}\wedge  \bar \Omega^n}{n!(n-1)!}\right)= \mathrm{tr}_\Omega(\partial_J( \alpha+ \beta))-|\alpha+\beta|^2
=\frac{s^{\mathrm{Ch}}}{2}-\frac{s^{\mathrm{Bis}}}{2}-|\alpha+\beta|^2\,,
\end{aligned}
\]
where we used q-realness of $ \partial_J \alpha $ and the fact that $\mathrm{tr}_\Omega(\partial J \bar \beta)=\mathrm{tr}_\Omega(\partial_J \beta)$. Therefore
\[
d^*\theta_g=2\Re ( d^*(\alpha+\beta))=s^{\mathrm{Ch}}-s^{\mathrm{Bis}}-2|\alpha+\beta|^2\,,
\]
giving the equivalence between \eqref{1gauduchon} and \eqref{scalarigau}. Finally, the equivalence with \eqref{deldelJgau} follows from $d^*(\alpha+ \beta)=-*\partial \partial_J(\frac{\Omega^{n-1}\wedge \bar \Omega^n}{n!(n-1)!})$.
\end{proof}

\noindent The following is an adaptation of Lemma \ref{lem:gaudu} to the balanced case. The proof is analogous.

\begin{lem}\label{Lem:3.13}
Let $ (M^n,\mathsf{H},g) $ be  a hyperhermitian manifold. The following are equivalent
\begin{enumerate}
\item $ \omega_I $ is balanced\,;
\item $ \omega_L $ is balanced for any $ L\in \mathsf{H} $\,;
\item $ \alpha+\beta=0 $\,;
\item $ \partial(\Omega^{n-1}\wedge \bar \Omega^n)=0 $.
\end{enumerate}
\end{lem}

The following lemma will be useful later.

\begin{lem}\label{Lem:scalaribalanced}
Let $ (M^n,\mathsf{H},g) $ be a compact hyperhermitian manifold with non-negative Bismut scalar curvature. Then $g$ cannot have negative Chern scalar curvature, furthermore $g$ is balanced if and only if it is Chern scalar flat. 
\end{lem}
\begin{proof}
Let $ \Omega_{\mathrm{G}}=\mathrm{e}^f\Omega $ be the $(2,0)$-form associated to the Gauduchon metric in the conformal class of $ \Omega $. Then
\[
\begin{split}
0&=s^{\mathrm{Ch}}(\Omega_{\mathrm{G}})-s^{\mathrm{Bis}}(\Omega_{\mathrm{G}})-2|\alpha_{\Omega_{\mathrm{G}}}+\beta_{\Omega_{\mathrm{G}}}|^2_{\Omega_{\mathrm{G}}}\\
&=\mathrm{e}^{-f}\left( s^{\mathrm{Ch}}(\Omega)-s^{\mathrm{Bis}}(\Omega)-2|\alpha_\Omega+\beta_\Omega+(2n-1)\partial f|^2_{\Omega} -2(2n-1)\Delta_{\Omega}f \right)
\end{split}
\]
Since $ \Omega $ has non-negative Bismut scalar curvature we deduce
\begin{equation}\label{eqn:Lemma3.9}
2(2n-1)\Delta_{\Omega}f\leq s^{\mathrm{Ch}}(\Omega)-2|\alpha_\Omega+\beta_\Omega+(2n-1)\partial f|^2_{\Omega}\,.
\end{equation}
If we had $ s^{\mathrm{Ch}}(\Omega)<0 $ the maximum principle would imply that $ f $ is constant yielding the inequality $ s^{\mathrm{Ch}}(\Omega)\geq 2|\alpha_\Omega+\beta_\Omega|^2_{\Omega} $ which contradicts the negativity of $ s^{\mathrm{Ch}}(\Omega) $. Furthermore, it is clear from \eqref{eqn:Lemma3.9} that if  $ s^{\mathrm{Ch}}(\Omega)=0 $  then  $ \alpha_\Omega+\beta_\Omega=0 $ and, hence, $\theta_g=0$. Conversely, if $\Omega$ is balanced, $s^{{\rm Ch}}(\Omega)=s^{{\rm Bis}}(\Omega)\ge 0 $. On the other hand, integrating by parts and using Lemma \ref{Lem:3.13}, we have
\[
\int_Ms^{{\rm Ch}}(\Omega)\frac{\Omega^n\wedge \bar \Omega^n}{(n!)^2}=2\int_M \partial_J\alpha_{\Omega}\wedge\frac{\Omega^{n-1}\wedge \bar \Omega^n}{n!(n-1)!}=0\,,
\]
but thus is possible only if $s^{{\rm Ch}}(\Omega)=0$, giving the claim.
\end{proof}

As mentioned in the introduction,  many other definitions of special hyperhermitian metrics, peculiar of the hypercomplex setting, were given. We shall recall their definition.

\begin{defn}\label{def:specialmetrics}
Let $ (M^n,\mathsf H ) $ be a hypercomplex manifold. A hyperhermitian metric $ g $ is:
\begin{enumerate}
  \item \emph{hyperk\"ahler with torsion} (HKT) if $ \partial \Omega=0 $ \cite{Howe-Papadopoulos (2000)}; 
  \item  \emph{quaternionic balanced} if $ \partial \Omega^{n-1}=0 $  \cite{Lejmi-Weber2};
  \item  \emph{quaternionic strongly Gauduchon} if $ \partial \Omega^{n-1} $ is $ \partial_J $-exact \cite{Lejmi-Weber}; 
  \item  \emph{quaternionic Gauduchon} if $\partial\partial_J\Omega^{n-1}=0$ \cite{GLV}. 
  \end{enumerate}
\end{defn}

\noindent All the definitions above are mimicking the analogous definitions of, respectively, K\"ahler, balanced, strongly Gauduchon, Gauduchon  metrics arising in complex Geometry, see \cite{cPaul,Gau2,Mich, Pop09}.\\
The properties in Definition \ref{def:specialmetrics} are listed in order of strength.
However, we shall collect  examples in   Section \ref{Sec:Ex} showing that the inclusions among the classes of manifolds admitting these metrics are strict. For instance, as far as we know, we provide the first example of compact hypercomplex manifold admitting quaternionic strongly Gauduchon metrics but no quaternionic balanced metrics.
\section{Proof of Theorem \ref{Thm:main1.2}} \label{Sec:THM1.2}

 First of all,  we give the definition of \emph{first quaternionic Bott-Chern class} of $(M, I, J)$. Let $ E $ be a $ I $-holomorphic line bundle over a hypercomplex manifold $ (M,\mathsf{H}) $. The curvature $ R_h $ of any Hermitian metric $h$ on $ E $ is a locally $\partial \bar \partial$-exact,  real $ (1,1) $-form on $ M $,  hence  $ \Phi(\frac{R_h-JR_h}{2}) $ is a q-real, $(2,0)$-form which is locally $\partial\partial_J$-exact. We then are led to give the following definition.
\begin{defn}
Let $(E, h)$ be a $I$-Hermitian line bundle with curvature $R$ over a hypercomplex manifold $(M, \mathsf H)$. The \emph{first quaternionic Bott-Chern class of $ E $ with respect to $ J $} is 
\[
c_1^{\mathrm{qBC}}(E,J):=\left[\Phi\left(\frac{R_h-JR_h}{2} \right)\right]_{\mathrm{qBC}}\in H^{2,0}_{{\rm qBC}}(M).
\]
Moreover, we call $ c_1^{\mathrm{qBC}}(M,I,J):=c_1^{\mathrm{qBC}}(K(M,I)^{-1},J) $ the \emph{first quaternionic Bott-Chern class of} $ M $ \emph{with respect to} $ (I,J) $.
\end{defn}
The class $c_1^{\mathrm{qBC}}(E,J)$ does not depend on the choice of $ h $. Indeed, any other Hermitian metric $h'$ on $ E $ has curvature $ R_{h'}=R_h+\sqrt{-1} \partial \bar \partial f $, for some $ f\in C^\infty(M,\R) $, hence,  thanks to \eqref{Lem:ddJ},  $ \Phi(\frac{R_{h'}-JR_{h'}}{2})=\Phi(\frac{R_h-JR_h}{2})+\frac{1}{2}\partial \partial_J f $. 
Observe that for any pair of $ I $-holomorphic line bundles $ E $ and $ F $ we have
\[
c_1^{\mathrm{qBC}}(E\otimes F,J)=c_1^{\mathrm{qBC}}(E,J)+c_1^{\mathrm{qBC}}(F,J)\,.
\]

As it is well-known, any hyperhermitian metric $ \omega_I $ on $ (M,\mathsf{H}) $, induces a Hermitian metric on $ K(M,I)^{-1} $ with curvature $ R=\Ric^{\mathrm{Ch}}_{\omega_I} $, hence, as a consequence of  Proposition \ref{Prop:alpha} (\ref{Lem:ricdelJalpha}), we have
$$
c_1^{\mathrm{qBC}}(M,I,J)=[\partial_J\alpha]_{{\rm qBC}}\,. 
$$

We shall give an equivalent definition of $c_1^{\mathrm{qBC}}(M,I,J)$. Let $ \Theta\in \Lambda^{2n,0}_IM $ be any q-positive section of the canonical bundle, then there exists $ \alpha_\Theta \in \Lambda^{1,0}_IM $ such that
\[
\partial \bar \Theta=\alpha_\Theta \wedge \bar \Theta\,.
\]
Similarly as before, $ [\partial_J \alpha_\Theta ]_{\mathrm{qBC}}$ does not depend on the choice of $ \Theta $. Indeed, if $ \Theta' $ is another q-positive $ (2n,0) $-form, there exists a function $ f\in C^\infty(M,\R) $ such that $ \Theta'=\mathrm{e}^f\Theta $.
Therefore,  $ \partial_J\alpha_{\Theta'}=\partial_J \alpha_\Theta -\partial \partial_J f $, which shows that $ [\partial_J \alpha_{\Theta'}]_{\mathrm{qBC}}=[\partial_J \alpha_\Theta]_{\mathrm{qBC}} $.

In particular,  we observe that $ c_1^{\mathrm{qBC}}(M,I,J)=0 $ if and only if there exists a metric in each hyperhermitian conformal class such that $ \partial_J \alpha=0 $, equivalently by  Proposition \ref{Prop:alpha} (\ref{Lem:ricdelJalpha}) the Chern-Ricci form of $ \omega_I $ is $ J $-invariant. Clearly, the definition of the first quaternionic Bott-Chern class depends on the choice of a basis $ (I,J) $ for the hypercomplex structure. However,  we can prove the following.
 
\begin{prop}\label{garullabarulla}
Let $ (M,\mathsf{H}) $ be a hypercomplex manifold. If $ c_1^{\mathrm{qBC}}(M,I,J) $ is positive, zero or negative, then so is $ c_1^{\mathrm{qBC}}(M,P,L) $,  for any other anticommuting complex structures $ L,P\in \mathsf H $.
\end{prop}
\begin{proof}
First of all, as done at the beginning of the proof of Proposition \ref{Prop:scalar},  it is enough to show the claim for $ P=I $ and $ L=aJ+bK\in \mathsf{H} $. Then, for any hyperhermitian metric $ \Omega $, keeping in mind the identities \eqref{eq:RicL}, we have
\begin{equation}\label{eqn:Prop4.2}
\frac{\Ric^{\mathrm{Ch}}_{\omega_L}-I\Ric^{\mathrm{Ch}}_{\omega_L}}{2}=(\Ric^{\mathrm{Ch}}_{\omega_L})^{2,0}+(\Ric^{\mathrm{Ch}}_{\omega_L})^{0,2}=(a-\sqrt{-1}b)\partial_J \alpha+ (a+\sqrt{-1}b)\overline{\partial_J \alpha}\,.
\end{equation}
If $ c_1^{\mathrm{qBC}}(M,I,J)=0 $,  we can choose $ \Omega $ so that $ \partial_J \alpha=0 $, which allows us to conclude that $ c_1^{\mathrm{qBC}}(M,I,L)=0 $. If instead $c_1^{\mathrm{qBC}}(M,I,J)>0$ we can choose $\Omega$ so that $\partial_J \alpha$ is q-positive, then, again from  \eqref{eqn:Prop4.2}, it follows that $\Ric^{\mathrm{Ch}}_{\omega_L}-I\Ric^{\mathrm{Ch}}_{\omega_L}$ is a positive $(1,1)$-form with respect to $L$, which implies that $c_1^{\mathrm{qBC}}(M,I,L)>0$, as claimed. The negative case is analogous.
\end{proof}
Thanks to Proposition \ref{garullabarulla}, when this occurs, we will unambiguously write $ c_1^{\mathrm{qBC}}(M,\mathsf{H})>0 $, $ c_1^{\mathrm{qBC}}(M,\mathsf{H})=0 $, or $ c_1^{\mathrm{qBC}}(M,\mathsf{H})<0 $.

\begin{rmk}\label{rmk:implications}
Let $ (M,\mathsf{H}) $ be a hypercomplex manifold.  We have that $c_1^{{\rm BC}}(M, I)=0$ forces $c_1^{{\rm qBC}}(M, \mathsf H)=0$. Indeed, if $ c_1^{\mathrm{BC}}(M,I)=0 $, for any hyperhermitian metric $ \Omega $ on $ (M,\mathsf{H}) $ we have that $ \Ric^{\mathrm{Ch}}_{\omega_I}$ is $\partial \bar \partial$-exact. Hence, by a conformal rescaling, we can find a hyperhermitian metric which is Chern-Ricci flat, whence $ c_1^{\mathrm{qBC}}(M,\mathsf{H})=0 $. We refer to Subsection \ref{Ex:Andrada-Tolcachier} for an example that shows that the converse does not hold, in general.
\end{rmk}
\smallskip
We observe that the vanishing of the first quaternionic Bott-Chern class restricts the possibilities for the Kodaira dimension, denoted as $ \kappa(M,L) $,  for any  $L\in \mathsf H$.
\begin{prop}\label{Prop:kod}
Let $ (M,\mathsf{H}) $ be a compact hypercomplex manifold. If $ c_1^{\mathrm{qBC}}(M,\mathsf{H})=0 $, then $ \kappa(M,L)\leq 0 $ for all $ L\in \mathsf{H} $. Moreover, $ \kappa(M,L)=0 $ if and only if $ K(M,L) $ is holomorphically torsion.
\end{prop}
\begin{proof}
Without loss of generality we prove the statement for $ L=I $. Let $ \Omega $ be any hyperhermitian metric on $ (M,\mathsf{H}) $. Since $ c_1^{\mathrm{qBC}}(K(M,I),J)=-c_1^{\mathrm{qBC}}(M,I,J)=0 $, there exists a Hermitian metric $ h $ on $ K(M,I) $ such that $ \Phi(\frac{R_h-JR_h}{2})=0 $. Now, for any $ k\geq 1 $ and any section $ \psi \in H^0(M,K(M,I)^{\otimes k}) $, a straightforward computation gives
\[
\Delta_{\omega_I} |\psi|^2=|\nabla \psi|^2-k|\psi|^2\mathrm{tr}_{\omega_I}(R_{h})=|\nabla \psi|^2-2k|\psi|^2\mathrm{tr}_{\Omega}\left(\Phi\left(\frac{R_h-JR_h}{2}\right)\right)=|\nabla \psi|^2\geq 0\,,
\]
where $|\cdot |^2$ and $\nabla$ are the pointwise squared norm and the Chern connection with respect to the metric $h^k$ induced on the power $ K(M,I)^{\otimes k} $, respectively. The strong maximum principle now implies that $|\psi|^2$ is constant, whence $\nabla \psi\equiv 0$. Consequently, for any $ k\geq 0 $ we have $\dim H^0(M,K_M^{\otimes k})\leq 1$ from which it follows $\kappa(M,I)\leq 0$.

Finally, we have $ \kappa(M,I)=0 $ if and only if there is at least a power $ k\geq 1 $ such that $ \dim H^0(M,K(M,I)^{\otimes k})=1 $, i.e. there exists a global holomorphic section of $ K(M,I)^{\otimes k} $, which is then parallel and nowhere vanishing.
\end{proof}
\noindent Be aware that, in general, for two different complex structures in $ \mathsf{H} $ the corresponding Kodaira dimensions need not be equal (see Subsection \ref{Ex:Andrada-Tolcachier}).
\smallskip

We now turn our attention to the ${\rm SL}(n, \H)$ condition. Let $ (M^n,\mathsf{H}) $ be a hypercomplex manifold. In \cite[Claim 1.2]{Verbitsky (2007)}, Verbitsky observed that being $\mathrm{SL}(n,\H)$ implies that $K(M, L)$ is holomorphically trivial, for any $L\in \mathsf H $, in particular $c_1^{\mathrm{qBC}}(M,\mathsf{H})=0$. This can also be deduced from the following result, which can be seen as a generalisation of \cite[Theorem 2.2]{Ivanov-Petkov}.

\begin{prop}\label{Prop:glob_slnH}
Let $ (M^n,\mathsf{H}) $ be a hypercomplex manifold. Then the following are equivalent:
\begin{enumerate}
	\item \label{globSLnH1} the holonomy group of the Obata connection is contained in $ \mathrm{SL}(n,\H) $;
	\item \label{globSLnH2} the $ (1,0) $-form $ \alpha_\Omega $ is $ \partial $-exact,  for any hyperhermitian metric $ \Omega $ on $ (M,\mathsf{H}) $;
	\item \label{globSLnH3} in any hyperhermitian conformal class there exists a unique (up to scaling) metric $ \Omega $ on $ (M,\mathsf{H}) $ such that $ \alpha_\Omega=0 $.
\end{enumerate}
\end{prop}
\begin{proof}
We know that  $ {\rm Hol}(\nabla^{\mathrm{Ob}})\subseteq  \mathrm{SL}(n,\H) $ if and only if there exists a global q-positive $ \nabla^{\mathrm{Ob}} $-parallel section $\Theta$. Observe that the canonical bundle of a hypercomplex manifold $ (M,\mathsf{H}) $ is always topologically trivialised by $ \Omega^n $. We can therefore identify $\Theta$ with a function $T\in C^\infty(M,\R)$, then $\nabla^{\mathrm{Ob}}T=dT+\eta T$, which vanishes if and only if $\eta=-d(\log T)$, i.e. the connection 1-form is exact. Thus \eqref{globSLnH1} is equivalent to \eqref{globSLnH2}. Obviously \eqref{globSLnH3} implies \eqref{globSLnH2}; to see the converse, suppose $ \alpha_\Omega=\partial f $ for some $ f\in C^\infty(M,\R) $, then the conformally rescaled metric $ \Omega_f=\mathrm{e}^{-\frac{f}{n}}\Omega $ satisfies $ \alpha_{\Omega_f}=0 $. The uniqueness of such a metric is clear.
\end{proof}

We now study some interesting consequences of the existence of quaternionic Gauduchon and quaternionic balanced metrics.

\begin{prop}\label{Prop:Prop5.2}
Let $(M^n,\mathsf{H},g)$ be a compact hyperhermitian manifold admitting a compatible quaternionic Gauduchon metric. Then, the following are equivalent:
\begin{enumerate}
\item $\alpha=0$;   
\item  $ \Ric_{\omega_L}^{\mathrm{Ch}}=0\,,$ for all $L\in \mathsf{H}$; 
\item $ \partial_J \alpha=0\,.$
\end{enumerate}
If in addition $g$ is quaternionic balanced, then it cannot have negative Chern scalar curvature and the conditions above are also equivalent to
\begin{enumerate}
\item[(4)] $s^{\mathrm{Ch}}=0$; 
\item[(5)] $g$  is balanced.
\end{enumerate}
\end{prop}
\begin{proof}
From Proposition \ref{Prop:alpha} \eqref{alphaChern}, we know that $ \alpha=0 $ always implies Chern-Ricci flatness which in turn implies $ \partial_J \alpha=0 $ by  Proposition \ref{Prop:alpha} (\ref{Lem:ricdelJalpha}). Therefore, we only need to show that $ \partial_J \alpha=0 $ implies $ \alpha=0 $. Let $ \tilde \Omega $ be a quaternionic Gauduchon metric on $ (M,\mathsf{H}) $. In general, we have that 
\begin{equation}\label{eqn:dedejomegan}
\partial \partial_J \bar \Omega^n
=\left(-\partial_J \alpha_\Omega +\alpha_\Omega \wedge J^{-1}\bar \alpha_\Omega \right)\wedge \bar \Omega^n\,.
\end{equation}
Hence, assuming $ \partial_J\alpha_\Omega=0 $,  wedging with $\frac{ \tilde \Omega^{n-1}}{n!(n-1)!} $ and integrating over $M$, we have 
\[
0=\int_M \frac{\partial \partial_J \tilde \Omega^{n-1}\wedge \bar \Omega^n}{n!(n-1)!}=\int_M \frac{\tilde \Omega^{n-1}\wedge \partial \partial_J \bar \Omega^n}{n!(n-1)!}=\int_M|\alpha_\Omega|^2_{\tilde \Omega}\, \frac{\tilde \Omega^n \wedge \bar \Omega^n}{(n!)^2}\,,
\]
showing that  $\alpha_\Omega=0$.
Now, if $g$ is quaternionic balanced the rest of the Proposition follows from Lemma \ref{Lem:scalaribalanced}.
\end{proof}

Now we are ready to prove Theorem \ref{Thm:main1.2} and  Corollary \ref{Cor:main1.3}.

\begin{proof}[Proof of Theorem \ref{Thm:main1.2} and Corollary \ref{Cor:main1.3}]
The necessity follows from \cite[Proposition 16]{GLV}, \cite[Claim 1.2]{Verbitsky (2007)} and Remark \ref{rmk:implications}.
Conversely, the assumption $ c_1^{\mathrm{qBC}}(M,\mathsf{H})=0 $ implies that any hyperhermitian conformal class contains a hyperhermitian metric such that $ \partial_J \alpha=0 $, therefore $ \alpha=0 $ thanks to Proposition \ref{Prop:Prop5.2} and then $ (M,\mathsf{H}) $ is $ \mathrm{SL}(n,\H) $, by Proposition \ref{Prop:glob_slnH}. This concludes the proof of Theorem \ref{Thm:main1.2}.

It is clear that the $ \partial \partial_J $-Lemma implies $ c_1^{\mathrm{qBC}}(M,\mathsf{H})=0 $ and if $ (M,\mathsf{H}) $ admits an HKT metric then the $ \partial \partial_J $-Lemma is implied by the $ \mathrm{SL}(n,\H) $ condition, see \cite[Theorem 6]{GLV}. Therefore Corollary \ref{Cor:main1.3} follows. 
\end{proof}



Note that the HKT assumption in Corollary \ref{Cor:main1.3} cannot be weakened, see Example \ref{ex:Lejmi-Weber}. In the last part of this section we collect some consequences of Theorem \ref{Thm:main1.2}.


\begin{cor}\label{Lem:deldelJlemma}
Let $ (M^n,\mathsf{H}) $ be a compact hypercomplex manifold  satisfying the $ \partial \partial_J $-Lemma. Then, the following are equivalent:
\begin{enumerate}
    \item \label{qsG1} $ (M,\mathsf{H}) $ is an $\mathrm{SL}(n,\H)$-manifold;
    \item \label{qsG2} There exists a quaternionic Gauduchon metric on $ (M,\mathsf{H}) $;
    \item \label{qsG3} There exists a quaternionic strongly Gauduchon metric on $ (M,\mathsf{H}) $.
    \end{enumerate}
\end{cor}
\begin{proof}
The equivalence between \eqref{qsG1} and \eqref{qsG2} follows from Theorem \ref{Thm:main1.2}. We only need to prove that \eqref{qsG2} implies \eqref{qsG3}. Let $ \Omega $ be a quaternionic Gauduchon metric. Hence,   $ \partial_J \Omega^{n-1} $ is $ \partial $-closed and $ \partial_J $-exact. Thanks to the $\partial\partial_J$-Lemma, it is therefore $ \partial \partial_J $-exact so that $ \Omega $ is quaternionic strongly Gauduchon.
\end{proof}

\noindent By conjunction of Lemma \ref{Lem:deldelJlemma}, Theorem \ref{Thm:main1.2} and \cite[Theorem 10.1]{Lejmi-Weber} we infer

\begin{cor}
Let $ (M^2,\mathsf{H}) $ be a compact hypercomplex manifold. Then,  any two of the following conditions imply the third:
\begin{enumerate}
\item\label{item:Cor5.61} $ (M,\mathsf{H}) $ is $ \mathrm{SL}(2,\H) $;
\item \label{item:Cor5.62} There exists an HKT metric on $ (M,\mathsf{H}) $;
\item\label{item:Cor5.63} The $ \partial \partial_J $-Lemma holds on $ (M,\mathsf{H}) $.
\end{enumerate}
\end{cor}

\noindent We should mention that, assuming \eqref{item:Cor5.61}, the combination of \cite[Theorem 25]{GLV} and \cite[Theorem 1]{GLV} already gives that \eqref{item:Cor5.62} and \eqref{item:Cor5.63} are equivalent.

Assuming the existence of a quaternionic balanced metric, we can strengthen Proposition \ref{Prop:kod}.
\begin{cor}\label{Cor:kod}
Let $ (M^n,\mathsf{H}) $ be a compact hypercomplex manifold admitting a compatible quaternionic balanced metric. Then, $ c_1^{\mathrm{qBC}}(M,\mathsf{H})=0 $ if and only if $ \kappa(M,L)=0 $ for all $ L\in \mathsf{H} $,  and $ c_1^{\mathrm{qBC}}(M,\mathsf{H})\neq 0 $ if and only if $ \kappa(M,L)=-\infty $,  for all $ L\in \mathsf{H} $.
\end{cor}
\begin{proof}
In these assumptions, we know from Theorem \ref{Thm:main1.2} that $ c_1^{\mathrm{qBC}}(M,\mathsf{H})=0 $ is equivalent to $ K(M,L) $ being holomorphically trivial for all $ L\in \mathsf{H} $ and, by Proposition \ref{Prop:kod}, this implies $ \kappa(M,L)=0 $,  for all $ L\in \mathsf{H} $.
On the other hand, if $ c_1^{\mathrm{qBC}}(M,\mathsf{H})\neq 0 $ the Gauduchon metric $ \Omega_{\mathrm{G}} $ in the conformal class of any quaternionic balanced metric cannot be balanced, otherwise the manifold would be $ \mathrm{SL}(n,\H) $. Indeed,
the form $ \beta_{\Omega_{\mathrm{G}}} $ would be $ \partial $-exact, but then, if $ \Omega_{\mathrm{G}} $ were balanced, $ \alpha_{\Omega_{\mathrm{G}}} $ would also be $ \partial $-exact, implying the $ \mathrm{SL}(n,\H) $ condition. Exploiting Bismut-Ricci flatness of quaternionic balanced metrics and \cite[Proposition 3.1]{Alexandrov-Ivanov}, we conclude.
\end{proof}

\begin{rmk}
It remains open to determine if a hypercomplex manifold $ (M,\mathsf{H}) $ with holomorphically trivial canonical bundle with respect to any $ L\in \mathsf{H} $ is necessarily $ \mathrm{SL}(n,\H) $, see \cite[Remark 6.4]{AT}.
\end{rmk}

\section{Existence results for quaternionic Gauduchon and quaternionic balanced metrics} \label{Sec:THM1.1}
Firstly, we rewrite the quaternionic Gauduchon condition in terms of the $ (1,0) $-form $ \beta $.

\begin{lem}\label{Lem:lemma5.1}
Let  $ (M^n,\mathsf{H},g) $ be a hyperhermitian manifold. Then, $ g $ is quaternionic Gauduchon if and only if  $ s^{\mathrm{Bis}}+2|\beta|^2=0 $.
\end{lem}
\begin{proof}
Expanding $\partial \partial_J\Omega^{n-1}$ and using \eqref{eqn:bismutscalarbeta}, we obtain the claim.
\end{proof}

\noindent Thanks to Proposition \ref{Prop:scalar} and Proposition \ref{Prop:alpha} \eqref{alphaLee}, we see that the quaternionic Gauduchon and the quaternionic balanced conditions do not depend on the pair of anticommuting complex structures in $ \mathsf{H} $. This is no longer true for quaternionic strongly Gauduchon metrics, see Subsection \ref{Subsec:qsG}.

As a first difference with the compact complex case, where Gauduchon metrics exist in any conformal class (see \cite{cPaul}), we emphasise that there are examples of compact hypercomplex manifolds which do not admit any quaternionic Gauduchon metric, see Subsection \ref{Ex:Andrada-Tolcachier}. However, as observed in the previous section,   a compact $ \mathrm{SL}(n,\H) $-manifold always admits quaternionic Gauduchon metrics. Similarly, we can prove:

\begin{lem}\label{Lem:lemma6.1}
Let $ (M^n,\mathsf{H}) $ be an $ \mathrm{SL}(n,\H) $-manifold. Then,  there exists a quaternionic balanced metric if and only if there exists a balanced hyperhermitian metric and the two metrics are conformal.
\end{lem}
\begin{proof}
By Lemma \ref{Lem:3.13}, a hyperhermitian metric $g $ is balanced if and only if $ \alpha+\beta=0 $ and so, by the $ \mathrm{SL}(n,\H) $ condition, $ \beta $ is $ \partial $-exact, thus, with a conformal change, we can find a quaternionic balanced metric. Conversely, if $\beta=0$, then $ \theta$ is $d$-exact, using the $\mathrm{SL}(n,\H)$-condition. Hence, we can find a hyperhermitian balanced metric performing a conformal change.
\end{proof}


We emphasise that quaternionic Gauduchon metrics exist on hypercomplex manifolds that are not $ \mathrm{SL}(n,\H) $, indeed there even exist HKT non-$ \mathrm{SL}(n,\H) $ manifolds such as the ones constructed by Joyce, see Subsection \ref{Subsec:Joyce},  or Swann \cite{Swann}. 

We will now prove Theorem \ref{Thm:main1.1}. 

\begin{proof}[Proof of Theorem \ref{Thm:main1.1}]
First of all, we focus on the quaternionic Gauduchon case. To begin with, we show that the conditions  \eqref{eqn:neccond1.2} and \eqref{eqn:neccond1.1} are necessary for the existence of a quaternionic Gauduchon metric in a given hyperhermitian conformal class. 
Let $\Omega=\mathrm{e}^{\frac{f}{n-1}}\Omega_{{\rm G}}$ be a quaternionic Gauduchon metric, where $\Omega_{\mathrm{G}}$ is the Gauduchon metric with unit volume in the conformal class of $\Omega$. Using \eqref{eqn:dedejomegan}, we have that
\[
  \begin{aligned}
  0=&\, \int_M\frac{\partial \partial_J\Omega^{n-1}\wedge \bar{\Omega}^n_{\mathrm{G}}}{n!(n-1)!}
  =\int_M\frac{\Omega^{n-1}\wedge \partial \partial_J\bar{\Omega}^n_{\mathrm{G}}}{n!(n-1)!}=-\int_M\left({\rm tr}_{\Omega}(\partial_J\alpha_{\Omega_{\mathrm{G}}})- \lvert \alpha_{\Omega_{\mathrm{G}}}\rvert^2_{\Omega}\right) \frac{\Omega^n\wedge \bar {\Omega}^n_{\mathrm{G}}}{(n!)^2}\,.
  \end{aligned}
\]
Then
\begin{equation}\label{nece1}
\int_M\mathrm{e}^{f}\left(s^{\mathrm{Ch}}(\Omega_{\mathrm{G}})- 2\lvert \alpha_{\Omega_{\mathrm{G}}}\rvert^2_{\Omega_{\mathrm{G}}}\right)\frac{\Omega^n_{\mathrm{G}}\wedge \bar{\Omega}^n_{\mathrm{G}}}{(n!)^2}=0\,.
\end{equation}
Identity \eqref{nece1} tells us that 
\begin{equation*}\label{veranece}
A=\left\{ f \in C^{\infty}(M, \mathbb R)\quad \middle| \quad \int_M\mathrm{e}^{f}\left(s^{\mathrm{Ch}}(\Omega_{\mathrm{G}})- 2\lvert \alpha_{\Omega_{\mathrm{G}}}\rvert^2_{\Omega_{\mathrm{G}}}\right)\frac{\Omega^n_{\mathrm{G}}\wedge \bar{\Omega}^n_{\mathrm{G}}}{(n!)^2}=0\right\}\ne \emptyset\,.
\end{equation*}
One can easily show that the above condition is equivalent to \eqref{eqn:neccond1.1}.
Furthermore, using Lemma \ref{Lem:lemma5.1}, $\Omega$ is quaternionic Gauduchon if and only if
   \begin{equation}\label{eqndarisolvere}
  \Delta_{\Omega_{\mathrm{G}}}f + |\beta_{\Omega_{\mathrm{G}}} + \partial f|^2+\frac{1}{2}s^{\mathrm{Bis}}(\Omega_{\mathrm{G}})=0\,.
  \end{equation}
Integrating \eqref{eqndarisolvere} with respect to the volume induced by $\Omega_{\mathrm{G}}$ gives \eqref{eqn:neccond1.2}.
Let us now show the uniqueness part. 
Assume $ \Omega $ and $ \Omega_f:=\mathrm{e}^{\frac{f}{n-1}} \Omega $ are both quaternionic Gauduchon metrics with unit volume, then
\[
\begin{aligned}
0=s^{\mathrm{Bis}}(\Omega_f)+2|\beta_{\Omega_f}|^2_{\Omega_f}=&\, \mathrm{e}^{-\frac{f}{n-1}}\left(s^{\mathrm{Bis}}(\Omega)+2 \Delta_{\Omega} f + 2|\beta_{\Omega} + \partial f|_{\Omega}^2 \right)\\
=&\, 2\mathrm{e}^{-\frac{f}{n-1}}\left(\Delta_{\Omega} f +2 \Re( g(\beta_{\Omega},\partial f))+|\partial f|_{\Omega}^2\right)\,.
\end{aligned}
\]
Since we have $ \Delta_{\Omega} f + 2 \Re( g(\beta_{\Omega},\partial f))=-|\partial f|_{\Omega}^2\leq 0 $, we can regard $ f $ as a supersolution of the linear equation
\[
\Delta_{\Omega} \phi + 2 \Re( g(\beta_{\Omega},\partial \varphi))= 0\,. 
\] Applying the minimum principle, we obtain that $f$ must be constant. On the other hand, the fact  that both $\Omega_f $ and $\Omega$ have unit volume  guarantees that $f=0$.

We now prove the existence part of the theorem. Let $ \Omega_{\mathrm{G}}$ be a Gauduchon metric of unit volume and fix $h\in A$. Up to addition of a constant, we may and do assume that  $h$ has zero mean with respect to $\Omega_{{\rm G}}$. We consider $\Omega_h=\mathrm{e}^{\frac{h}{n-1}}\Omega_{\mathrm{G}}$ and, as done above, we have that 
\begin{equation}\label{eqn:PDEh}
\mathrm{e}^{\frac{h}{n-1}}\left(\frac12s^{\mathrm{Bis}}(\Omega_h)+\lvert\beta_{\Omega_h}\rvert^2_{\Omega_h}\right)=\frac12s^{\mathrm{Bis}}(\Omega_{\mathrm{G}})+\Delta_{\Omega_{\mathrm{G}}}h+ \lvert\beta_{\Omega_{\mathrm{G}}}+ \partial h\rvert^2_{\Omega_{\mathrm{G}}}\,.
\end{equation}
To solve \eqref{eqndarisolvere}, we perform the classical method of continuity.
We consider the following family of equations for $ t\in [0,1] $
\begin{equation}\label{continuity}
\Delta_{\Omega_{\mathrm{G}}} f+\lvert\beta_{\Omega_{\mathrm{G}}}+\partial f\rvert^2_{\Omega_{\mathrm{G}}}+  \frac{1}{2}s^{\mathrm{Bis}}(\Omega_{\mathrm{G}}) = (1-t)\mathrm{e}^{\frac{h}{n-1}}\left(\frac{1}{2}s^{\mathrm{Bis}}(\Omega_h)+ \lvert\beta_{\Omega_h}\rvert^2_{\Omega_h}\right)\,.
\end{equation}
We will search for solutions in $A$ with zero mean with respect to $\Omega_{{\rm G}}$. 
For $ t=0 $, we easily observe that $ h $ is a solution.
Now, consider $ t\in [0,1] $ such that the corresponding equation admits a solution $ f\in C^{2,\alpha}(M,\R) $ and define the operator $ F_t\colon A^{2,\alpha}_0 \to \R $ such that
\[
F_t(\varphi):=\Delta_{\Omega_{\mathrm{G}}} \varphi+\lvert\beta_{\Omega_{\mathrm{G}}}+\partial \varphi\rvert^2_{\Omega_{\mathrm{G}}}+  \frac{1}{2}s^{\mathrm{Bis}}(\Omega_{\mathrm{G}}) - (1-t)\mathrm{e}^{\frac{h}{n-1}}\left(\frac{1}{2}s^{\mathrm{Bis}}(\Omega_h)+ \lvert\beta_{\Omega_h}\rvert^2_{\Omega_h}\right)\,,
\]
for any $t\in [0,1]$, where 
\[
A^{2, \alpha}_0=\left\{\varphi\in C^{2, \alpha }_0(M, \R) \quad \middle| \quad \int_M\mathrm{e}^{\phi}\left(s^{\mathrm{Ch}}(\Omega_{\mathrm{G}})- 2\lvert \alpha_{\Omega_{\mathrm{G}}}\rvert^2_{\Omega_{\mathrm{G}}}\right)\frac{\Omega^n_{\mathrm{G}}\wedge \bar{\Omega}^n_{\mathrm{G}}}{(n!)^2}=0\right\}\,.
\]
Here $ C^{2,\alpha}_0(M,\R) $ denotes the functions in $ C^{2,\alpha}(M,\R) $ with zero mean with respect to $ \Omega_{\mathrm{G}} $.
Then, the linearisation of $F_t$ at $ f $ is 
\[
d_fF_t(v)=\Delta_{\Omega_{\mathrm{G}}} v +2\Re( g(\beta_{\Omega_{\mathrm{G}}}+\partial f, \partial v))\,, \quad v\in T_fA_0^{2, \alpha}\,.
\]
In particular, $d_fF_t$ has zero kernel on $T_fA_0^{2, \alpha}$, by the maximum principle.
On the other hand,  its index is equal to that of the Laplacian, implying that it is invertible. Applying the Implicit Function Theorem, we can conclude openness of the set of $ t\in [0,1] $ for which \eqref{continuity} is solvable. 

Now, we need to show closedness. First of all, we prove an $ L^2 $-gradient estimate. Integrating \eqref{continuity} and using \eqref{eqn:PDEh} we get:
\[
\lVert\beta_{\Omega_{\mathrm{G}}}+ \partial f\rVert^2_{L^2(\Omega_{\mathrm{G}})}=-\frac{t}{2}\Gamma^{\mathrm{Bis}}(\{\Omega_{\mathrm{G}}\}) + (1-t)\lVert\beta_{\Omega_{\mathrm{G}}}+ \partial h\rVert^2_{L^2(\Omega_{\mathrm{G}})}\leq C\,,
\]
thanks to the fact that $h$ is a datum. We may and do assume that $ \lVert\partial f\rVert_{L^2(\Omega_{\mathrm{G}})}^2 \geq ||\beta_{\Omega_{\mathrm{G}}}||_{L^2(\Omega_{\mathrm{G}})}^2 $, otherwise we are done. Consequently
\[
\|\partial f \|^2_{L^2(\Omega_{\mathrm{G}})}- \|\beta_{\Omega_{\mathrm{G}}}\|^2_{L^2(\Omega_{\mathrm{G}})}=\left|\|\partial f \|^2_{L^2(\Omega_{\mathrm{G}})}- \|\beta_{\Omega_{\mathrm{G}}}\|^2_{L^2(\Omega_{\mathrm{G}})}\right| \leq \|\partial f+\beta_{\Omega_{\mathrm{G}}}\|^2_{L^2(\Omega_{\mathrm{G}})} \leq C\,,
\]
which gives the desired estimate
\begin{equation}\label{eq:gradientL2}
\|\partial f\|^2_{L^2(\Omega_{\mathrm{G}})}\leq C\,.
\end{equation}
Using the Poincar\'{e} inequality we then get an estimate on the $L^2$-norm of $ f $. Now, setting
\[
\psi:=- \left( \frac{1}{2}s^{\mathrm{Bis}}(\Omega_{\mathrm{G}}) +|\beta_{\Omega_{\mathrm{G}}}|^2_{\Omega_{\mathrm{G}}}  +|\partial f|_{\Omega_{\mathrm{G}}}^2\right)+ (1-t)\mathrm{e}^{\frac{h}{n-1}}\left(\frac{1}{2}s^{\mathrm{Bis}}(\Omega_h)+ \lvert\beta_{\Omega_h}\rvert^2_{\Omega_h}\right)
\]
and rearranging equation \eqref{continuity},  we can  consider $ f $ as a solution of the linear equation
\[
\Delta_{\Omega_{\mathrm{G}}} \phi+2\Re( g(\beta_{\Omega_{\mathrm{G}}}, \partial \phi))=\psi\,.
\]
This allows us to  use the Calder\'{o}n-Zygmund inequality \cite[Theorem 9.11]{Gilbarg-Trudinger} to deduce
\[
\|f\|_{W^{2,2}(\Omega_{\mathrm{G}})} \leq C \left( \| f\|_{L^2(\Omega_{\mathrm{G}})} + \|\psi\|_{L^2(\Omega_{\mathrm{G}})} \right) \leq C\,,
\]
thanks to \eqref{eq:gradientL2}. We can obtain higher-order estimates by bootstrapping. Using Sobolev embeddings gives us the estimates we were looking for, concluding the existence part of the theorem.

As regards quaternionic balanced metrics, let $\Omega$ be a quaternionic balanced of unit volume  in the conformal class of $\Omega_{{\rm G}}$. Then, it is, in particular, quaternionic Gauduchon and thus  \eqref{eqn:neccond1.1}  holds. On the other hand, quaternionic balanced metrics are Bismut-Ricci flat by Proposition \ref{Prop:alpha} (\ref{alphaBismut}). Hence, the Bismut-Ricci form of $\Omega_{{\rm G}}$ with respect to $ I $ is $\partial\bar \partial $-exact, implying   $\Gamma^{{\rm Bis}}(\{\Omega_{{\rm G}}\})=0.$
Conversely, using the previous part of the theorem, we can find a solution $f\in C^{\infty}(M, \R)$ to  \eqref{eqndarisolvere}. On the other hand, integrating again \eqref{eqndarisolvere} and using that $\Gamma^{\mathrm{Bis}}(\{\Omega_{{\rm G}}\})=0$ we obtain that $\beta_{\Omega_{{\rm G}}}=-\partial f$. This gives us the claim. 
\end{proof}

\begin{rmk}
If $ (M^n,\mathsf{H}) $ is a compact $ \mathrm{SL}(n,\H) $-manifold we can recover  \cite[Proposition 16]{GLV} as a corollary of Theorem \ref{Thm:main1.1}. Indeed, the $ \mathrm{SL}(n,\H) $ condition implies that $ \alpha_{\Omega_{\mathrm{G}}}=\partial f $, for some $ f\in C^\infty(M,\R) $, where $ \Omega_{\mathrm{G}} $ is any Gauduchon hyperhermitian metric on $ (M,\mathsf{H}) $. The metric $ \Omega_{\mathrm{G}} $ then satisfies
\[
s^{\mathrm{Bis}}(\Omega_{\mathrm{G}})=s^{\mathrm{Ch}}(\Omega_{\mathrm{G}})-2|\alpha_{\Omega_{\mathrm{G}}} + \beta_{\Omega_{\mathrm{G}}} |^2_{\Omega_{\mathrm{G}}}=-2\Delta_{\Omega_{\mathrm{G}}} f-2|\partial f + \beta_{\Omega_{\mathrm{G}}} |^2_{\Omega_{\mathrm{G}}}
\]
which, integrated gives \eqref{eqn:neccond1.2}. 
Also, suppose $ s^{\mathrm{Ch}}(\Omega_{\mathrm{G}})-2|\alpha_{\Omega_{\mathrm{G}}}|^2_{\Omega_{\mathrm{G}}}=-2\Delta_{\Omega_{\mathrm{G}}} f-2|\partial f|^2_{\Omega_{\mathrm{G}}} $ has a sign, then the same is true for
\[
\mathrm{e}^{f}\left(|\partial f|^2_{\Omega_{\mathrm{G}}}+\Delta_{\Omega_{\mathrm{G}}} f\right)=\Delta_{\Omega_{\mathrm{G}}} (\mathrm{e}^{f})\,.
\]
By the maximum principle $ s^{\mathrm{Ch}}(\Omega_{\mathrm{G}})-2|\alpha_{\Omega_{\mathrm{G}}}|^2_{\Omega_{\mathrm{G}}} $ must vanish identically, so that \eqref{eqn:neccond1.1} is satisfied.
\end{rmk}

 
One can observe that the search for a quaternionic Gauduchon metric within a conformal class can be considered as a particular instance of the problem of prescribing the Bismut scalar curvature on a compact complex manifold. This problem was  taken into account only in the constant case, i.e. the Bismut-Yamabe problem, in \cite{Barbarone}. On the other hand, many results about the related problem of prescribing the Chern scalar curvature  can be found in the literature, see for instance \cite{ACS, bolo, yu} and the references therein. 

We now give another characterisation of the existence of quaternionic Gauduchon and quaternionic balanced metrics in terms of currents. 


\begin{prop}\label{qPaulcurrent}
Let $(M^n, \mathsf{H})$ be a compact hypercomplex manifold with $n\ge 2$. Then, $M$ admits no quaternionic Gauduchon metric if and only if there exists a non-zero,  $\partial \partial_J$-exact, q-real and q-positive $(2,2n)$-current. Furthermore, $M$ admits no quaternionic balanced metrics if and only if there exists a non-zero,  $\partial $-exact, q-real and q-positive $(2,2n)$-current.
\end{prop}
\begin{proof}
Let us begin with quaternionic Gauduchon metrics.
We consider the following spaces:
\[
W_1=\{\phi \in \Lambda^{2n-2,0}_IM \mid J\bar \phi=\phi\,,\,  \phi\mbox{ q-positive}\}\,,\quad W_2=\{\psi\in \Lambda_I^{2n-2,0}M \mid \partial\partial_J \psi=0\}\,.
\] 
Using Lemma \ref{Lem:bijection}, every $\phi\in W_1$ can be written as $\phi=\Omega^{n-1}$ for some hyperhermitian metric $\Omega$. Then,  the non existence of quaternionic Gauduchon metrics on $M$ is equivalent to  
 $W_1\cap W_2=\emptyset$. Thus, thanks to the Hahn-Banach Theorem, see for instance \cite[Theorem 31]{GLV}, we can find a current $T\in \mathcal D_I^{2,2n}(M)$  such that 
\[
T|_{W_1}>0\,, \quad T|_{W_2}=0\,.
\]
The fact that $T|_{W_1}>0$  guarantees that $T$ is both q-real and q-positive.
On the other hand, since $T|_{W_2}=0$,  we infer that 
$\partial T=0$ and $\partial_JT=0$.
Hence,  $T$ defines a cohomology class
\[
[T]_{{\rm qBC}}\in H_{{\rm qBC}}^{'2,2n}(M, \R):=\frac{ \{ T\in \mathcal D_I^{2,2n}(M) \mid \partial T=\partial_JT=0\}}{\partial\partial_J \mathcal D_I^{0,2n}(M) }\,.
\]
Now, the claim is equivalent to  prove that $[T]_{{\rm qBC}}=0$. In order to do this, we identify $T$ with an element in $\Lambda^{2,2n}_IM\otimes \mathcal D^0(M) $, where $\mathcal D^0(M)$ are the distributions on $M$.  This identification is compatible with $\partial$ and $\partial_J$. Furthermore, adapting the proof in \cite[pp. 1151--1152]{GLV} with appropriate bidegrees, one can see that $H_{{\rm qBC}}^{'2,2n}(M, \R)\simeq H^{2,2n}_{{\rm qBC}}(M, \R)$.
Moreover, the pairing
\[
\langle\cdot, \cdot\rangle\colon H^{2,2n}_{{\rm qBC}}(M)\times H^{2n-2,0 }_{{\rm qA}}(M)\to \mathbb R \quad \quad \langle[\phi]_{{\rm qBC}}, [\psi]_{{\rm qA}}\rangle=\int_M\phi\wedge \psi
\]
is well-defined and non-degenerate, which follows from adapting the arguments in \cite[Section 2]{Sch} with suitable Laplacians, whose explicit expression can be found in \cite{GLV}.
Now,  since $\langle[T]_{{\rm qBC}}, [\psi]_{{\rm qA}}\rangle=0$, for any $\partial \partial_J$-closed $(2n-2, 0)$-form $\psi$, then $[T]_{{\rm qBC}}=0$, proving the claim.

Regarding quaternionic balanced metrics, the strategy is entirely the same, it is sufficient to choose $W_1$ as before, replace $W_2$ with
\[
W_2=\{\psi\in \Lambda_I^{2n-2,0}M \mid \partial \psi=0\}\,,
\]
and use the non-degenerate pairing 
\[
\langle\cdot, \cdot\rangle\colon H^{2, 2n }_{\partial}(M)\times H^{2n-2,0}_{\partial}(M)\to \mathbb R\,, \quad \quad \langle[\phi], [\psi]\rangle=\int_M\phi\wedge \psi
\] to conclude.
\end{proof}

\noindent In Subsection \ref{qsGnoqb} we shall use Proposition \ref{qPaulcurrent} to provide examples of compact quaternionic strongly Gauduchon manifolds not admitting quaternionic balanced metrics.

 As highlighted by Theorem \ref{Thm:main1.1} and Example \ref{Ex:Andrada-Tolcachier}, it is not always possible to find quaternionic Gauduchon metrics on a given compact hypercomplex manifold. 
This can  also be understood  from Proposition \ref{qPaulcurrent}. Chosen $\Theta\in \Lambda_{I}^{2n, 0}M$ to be a q-positive volume form on $M$, any  $\partial \partial_J$-exact, q-real and q-positive $(2,2n)$-current can be written as
\begin{equation}\label{eq:positivecurrent}
  T=\partial \partial_J(f\bar \Theta)=\partial \partial_J f \wedge \bar \Theta+  \partial f \wedge \partial_J \bar \Theta- \partial_J f \wedge \partial \bar \Theta+ f\partial \partial_J \bar \Theta\,, \quad  f\in \mathcal D^0(M, \mathbb C)\,.
\end{equation}
  Note that the distribution $f$ might be complex valued. However, if we further impose $\partial\partial_J\bar \Theta=0$, then since $T$ is q-real, we obtain that $h=-\frac{\sqrt{-1}}{2}(f-\bar f)\in \mathcal D^{0}(M, \mathbb R)$ satisfies 
  \[
  \partial\partial_J h \wedge \bar \Theta+ \partial h \wedge \partial_J \bar \Theta- \partial_Jh \wedge \partial \bar \Theta=0\,.
  \]
Then, $h$ satisfies a second order elliptic equation without zero order terms. So, applying standard elliptic regularity, $h\in C^{\infty}(M, \mathbb R)$, and by the maximum principle it is constant. On the other hand, if $c$ is a constant, then, of course,  $\partial\partial_J (c\bar \Theta)=0$ so, up to adding a suitable constant, we may assume that $f$ is real-valued. 

  Stemming from this discussion, we can give a weaker sufficient condition for the existence of quaternionic Gauduchon metrics.
  \begin{prop}\label{volformqG}
  Let $(M^n, \mathsf H )$ be a compact hypercomplex manifold. If there exists a q-positive volume form $\Theta\in \Lambda_I^{2n, 0}M$ such that $\partial\partial_J\bar \Theta=0$, then, there exists a quaternionic Gauduchon metric compatible with $\mathsf{H}$.
  \end{prop}
\begin{proof}
Thanks to Proposition \ref{qPaulcurrent}, the existence of a quaternionic Gauduchon metric is equivalent to the fact that any $\partial \partial_J$-exact, q-real and q-positive $(2,2n)$-current is zero. Now, as above, any such current $T$ can be written as in \eqref{eq:positivecurrent} and we may assume that $f\in \mathcal{D}^0(M,\R)$. Consider the following closed convex cone
\begin{equation}\label{conepositive}
C=\{f\in \mathcal D^0(M, \mathbb R)\quad |\quad \partial\partial_J(f\bar \Theta)\ge 0 \}\,.
\end{equation}
It is not  hard to prove that $C\cap C^{\infty}(M, \R)$ is dense in $C$ in the weak sense of distributions. So, fixed $f\in C$, we can find $\{f_n\}_n\subseteq C\cap C^{\infty}(M,\R)$ such that $f_n\to f$ in the weak sense of distributions.  On the other hand, $f_n$ satisfies 
\[
\partial \partial_J f_n \wedge \bar \Theta+ \partial f_n \wedge \partial_J \bar \Theta- \partial_J f_n \wedge \partial \bar \Theta \ge 0 
\]
and so we can apply the maximum principle obtaining that $f_n$ is constant. This implies that $f$ is constant,  giving the claim. 
\end{proof}

We remark that the hypothesis of q-positivity of $ \Theta$ cannot be removed.  Indeed, we can consider a non-${\rm SL}(n, \mathbb H)$-manifold $(M,\mathsf H)$ with $K(M, I)$ holomorphically trivial.  Then, we can find an holomorphic volume form $ \Theta$ which is, in particular, $\bar \partial \bar \partial_J$-closed but not q-positive, since the manifold is not ${\rm SL}(n, \mathbb H)$. On the other hand, Proposition \ref{volformqG} will  guarantee the existence of a quaternionic Gauduchon metric. Then, applying Theorem \ref{Thm:main1.2}, we will obtain that $(M, \mathsf H  )$ is ${\rm SL}(n, \mathbb H)$, reaching a contradiction.  An example of  a non-${\rm SL}(n, \mathbb H)$-manifold with holomorphically trivial canonical bundle can be found in Subsection \ref{Ex:Andrada-Tolcachier}.

\smallskip
By a very well-know result, see \cite[Proposition 1.9]{Mich}, the class of compact balanced manifolds is closed under products and proper holomorphic submersions. We conclude this section studying  the same closedness properties for the class of compact quaternionic balanced manifolds. 
First of all, we recall the definition of hypercomplex submersion, see for instance \cite{Alekseevsky-Marchiafava, IMV2, IMV}.
\begin{defn}
Let $(M, \mathsf H)$ and $(M', \mathsf H')$ be two hypercomplex manifolds. A  map  $f\colon M\to M'$ is called \emph{hypercomplex}  if,  for any $L\in\mathsf H$,  there exists $L'\in\mathsf H'$ such that $f\colon (M,  L )\to (M', L')$ is holomorphic.  If $f$ is also a submersion, then $f$ will be called a \emph{hypercomplex submersion}. 
\end{defn}
  Examples of hypercomplex submersions between compact manifolds can be produced standardly looking, for instance, at finite coverings of compact hypercomplex manifolds. 
 Another example can be found in \cite[Section 4]{IMV}. 
\begin{prop}
 Let $(M, \mathsf H )$  and  $(M', \mathsf H')$ be two quaternionic balanced manifolds and $(M'', \mathsf H'')$ be a hypercomplex manifold. Then, $(M\times M', \mathsf H\oplus\mathsf H')$ is quaternionic balanced. Moreover, if $f\colon (M, \mathsf H )\to (M'', \mathsf H'' )$ is a  proper hypercomplex submersion, then $(M'', \mathsf H'')$ is quaternionic balanced.  \end{prop}
\begin{proof}
The first claim is trivial by considering $\Omega=\Omega_M+ \Omega_{M'}$, where $\Omega_M $ and $\Omega_{M'}$ are quaternionic balanced metrics, respectively,  on $(M, \mathsf H)$ and $(M', \mathsf H')$. 

Let us now prove the second assertion. Let $ n $ and $ m $ be the quaternionic dimensions of $ M $ and $ M'' $ respectively. We fix $\Omega_M$ to be a quaternionic balanced metric on $ M $.
For the sake of simplicity, we will also fix basis $ (I,J) $ and $ (I'',J'') $ for the hypercomplex structures $\mathsf{H}$ and $ \mathsf{H}'' $, respectively, such  that $f$ is both $(I, I'')$ and $(J, J'')$-holomorphic. Then, since $f$ is proper, we can consider the pushforward
\[
 \gamma=f_{*}\Omega_M^{n-1}\,.
 \]
  Now, since $f$ is $(I, I'')$-holomorphic and $\partial\Omega_M^{n-1}=0$, then $\gamma\in\Lambda_{I''}^{2m-2,0}M''$ and $\partial\gamma=0 $. Thus, thanks to Lemma \ref{Lem:bijection}, it is sufficient to prove that $\gamma$ is q-real and q-positive. First of all, we note that q-realness is trivial because $f$ is $(J, J'')$-holomorphic and then $J''f_*=f_*J$.
  Using an adaptation to the hypercomplex setting of the argument in \cite[Proof of Proposition 1.9]{Mich}, one can easily see that $\gamma$ is q-positive, which allows us to conclude.
\end{proof}

\noindent A direct consequence of the second statement is that if $f\colon M \to M''$ is a finite covering, then $M$ is quaternionic balanced if and only if $M''$ is quaternionic balanced.


\section{Proof of Theorem \ref{Thm:main1.4}}\label{Sec:sHKT}
In this section, we will study properties of compact strong HKT manifolds establishing Theorem \ref{Thm:main1.4}.
First of all, we prove two preliminary formulae which hold in the hyperhermitian setting. One of these can be considered as  the quaternionic analogue of \cite[Formula (2.13)]{Alexandrov-Ivanov}.

 \begin{prop}\label{sHKT}
 Let $(M^n, \mathsf H , g)$ be a hyperhermitian manifold with $ n\geq 2 $. Then, for any $Z\in T^{1,0}_IM$, we have 
 \begin{equation}\label{eq:formulapartialJalpha}
\partial_J \alpha(Z, J \bar Z)=|\iota_Z \partial \bar \Omega |^2 + |\iota_{J\bar Z} \partial \bar \Omega|^2- n\frac{\iota_{J\bar Z}\iota_Z \left( \partial \partial_J \bar \Omega  \right)\wedge \bar \Omega^{n-1}}{\bar \Omega^n}\,.
\end{equation}
Moreover,
\begin{equation}\label{eq:sHKT}
\frac{1}{2}s^{\mathrm{Ch}}(\Omega)+g(\partial \partial_J \bar \Omega, \Omega \wedge \bar \Omega)-|\partial \bar \Omega|^2=0\,.
\end{equation}

 \end{prop}
 \begin{proof}
 By \eqref{eqn:dedejomegan}, we have that
\[
\left(-\partial_J \alpha +\alpha \wedge J^{-1}\bar \alpha\right)\wedge \bar \Omega^n=\partial \partial_J \bar \Omega^n=n\partial \partial_J \bar \Omega \wedge \bar \Omega^{n-1}+n(n-1)\partial \bar \Omega \wedge \partial_J \bar \Omega \wedge \bar \Omega^{n-2}\,.
\]  
Now, we fix $ Z\in T^{1,0}_IM $ and compute
\[
\iota_{J\bar Z}\iota_Z \left(\partial \bar \Omega \wedge \partial_J \bar \Omega \wedge \bar \Omega^{n-2} \right)= (\iota_{ Z}\partial \bar \Omega) \wedge (\iota_{J\bar Z}\partial_J\bar  \Omega) \wedge \bar \Omega^{n-2} -(\iota_{ J\bar Z}\partial \bar \Omega) \wedge (\iota_{Z}\partial_J \bar \Omega) \wedge \bar \Omega^{n-2}\,. 
\]
Using \eqref{eqn:hodgerel} and  Lemma \ref{Lem:alphabeta}, we infer that 
\[
\begin{split}
\iota_{J\bar Z}\iota_Z \left(\partial \bar \Omega \wedge \partial_J \bar \Omega \wedge \bar \Omega^{n-2} \right)
=\,&\frac{1}{n(n-1)}\left( |\alpha(Z)|^2+|\alpha(J\bar Z)|^2 - |\iota_Z \partial \bar \Omega |^2 - |\iota_{J\bar Z} \partial \bar \Omega|^2\right)\bar  \Omega^{n}\\
=\,&\frac{1}{n(n-1)}\left( (\alpha \wedge J^{-1}\bar \alpha) (Z,J\bar Z) - |\iota_Z \partial \bar \Omega |^2 - |\iota_{J\bar Z} \partial \bar \Omega|^2\right) \bar \Omega^{n}\,,
\end{split}
\]
 giving \eqref{eq:formulapartialJalpha}. Finally, \eqref{eq:sHKT} follows from \eqref{eq:formulapartialJalpha} by taking traces with respect to $\Omega$.
\end{proof}


As explained in the introduction, Theorem \ref{Thm:main1.4} goes in the direction of the so-called Fino-Vezzoni conjecture, see \cite[Problem 3]{FV}. Recall that a SKT (or pluriclosed) metric is a Hermitian metric such that the torsion of its Bismut connection is $d$-closed.
\noindent Then, the Fino-Vezzoni conjecture asserts that a compact complex manifold admitting both SKT and balanced metrics is necessarily K\"ahler.\\
 Recall that an HKT metric $g$ is said to be strong if the torsion of the Bismut connection of $g$ is closed. Hence, it is clear that an HKT metric $g$  is strong if and only if $\omega_L$ are all simultaneously SKT, for $ L\in \mathsf{H} $ or, equivalently, if $\partial\partial_J\bar \Omega=0$ \cite[Proposition 5.4]{Verbitsky (2009)}. Hence, the question of whether strong HKT and balanced hyperhermitian metrics can coexist on a non-hyperk\"ahler manifold is a particular instance of the Fino-Vezzoni conjecture.

Before we prove the announced Theorem \ref{Thm:main1.4} we need the following preliminary result.

\begin{prop}\label{Prop:FV}
Let $ (M^n,\mathsf{H}, g) $ be a  non-hyperk\"ahler,  strong HKT manifold with $ n\geq 2 $. Then $ \partial_J \alpha_\Omega\geq 0 $ and $ \partial_J \alpha_\Omega\neq 0 $. In particular, $ c_1^{\mathrm{qBC}}(M,I,J) $ admits a q-semipositive representative and, if $ M $ is compact, it is non zero. Hence, a non-hyperk\"ahler, strong HKT compact manifold cannot be ${\rm SL}(n, \H)$.
\end{prop}
\begin{proof}
From \eqref{eq:formulapartialJalpha}, using that $\partial\partial_J\bar \Omega=0$,  we see that
\[
\partial_J \alpha_\Omega (Z,J\bar Z)=|\iota_Z \partial \bar \Omega |^2 + |\iota_{J\bar Z} \partial \bar \Omega|^2 \geq 0\,, \quad Z\in T^{1,0}_IM
\]
and, since $ \Omega $ is not hyperk\"ahler,  there exists at least one $ Z \in T_I^{1,0}M$ such that 
$
\partial_J\alpha_{\Omega}(Z, J\bar Z)>0\,, 
$  proving the first statement.\\
Furthermore, suppose $ M $ is compact and $c_1^{{\rm qBC}}(M, \mathsf H)=0$, then there would be a function $f\in C^{\infty}(M, \mathbb R)$ such that 
$
0\le \partial_J\alpha_{\Omega}=\partial \partial_Jf. 
$ By the maximum principle, we infer that $\partial_J\alpha_{\Omega}=0$ which is impossible, thanks to the first part of the proof. The last part of the statement follows straightforwardly from the fact that the ${\rm SL}(n, \H)$-condition forces $c_1^{{\rm qBC}}(M, \mathsf H)=0$.
\end{proof}
We are now ready to prove Theorem \ref{Thm:main1.4}. 


\begin{proof}[Proof of Theorem \ref{Thm:main1.4}]
Let us suppose by contradiction that there exists  a balanced hyperhermitian metric $\Omega$. Under the assumptions of the theorem,  we can  choose another hyperhermitian metric $\tilde{\Omega}$ such that  $\partial_J\alpha_{\tilde{\Omega}}\ge 0$ but $\partial_J\alpha_{\tilde{\Omega}}\ne 0$. Now, we have that 
$$
c_1^{{\rm qBC}}(M, \mathsf H)\cdot[\Omega^{n-1}\wedge \bar \Omega^n]_{\partial}=\int_M\partial_J\alpha_{\tilde{\Omega}}\wedge\Omega^{n-1}\wedge \bar \Omega^n=0  \,,
$$integrating by parts and using the balanced condition.
This holds if and only if ${\rm tr}_{\Omega}(\partial_J\alpha_{\tilde{\Omega}})=0$, using q-semipositivity, which is equivalent to $\partial_J\alpha_{\tilde{\Omega}}=0,$ against the fact that $c_1^{{\rm qBC}}(M,  \mathsf H)\ne0.$ For $n=1$ the last claim now follows from the fact that balanced metrics are K\"ahler, while, for $n\ge 2$, it follows from Proposition \ref{Prop:FV} and the first part of the theorem.
\end{proof}
From Corollary \ref{Cor:kod} we also deduce the following fact about the Kodaira dimension.

\begin{cor}
Let $ (M,\mathsf{H}) $ be a hypercomplex manifold admitting a compatible strong HKT metric. Then,  $ \kappa(M,L)=-\infty $,  for all $ L\in \mathsf{H} $.
\end{cor}
\section{Chern-Einstein hyperhermitian metrics}\label{Sec:Einst}

As mentioned in the introduction, it is natural to define the following Einstein condition for a hyperhermitian metric. 


\begin{defn}
Let $ (M^n,\mathsf{H}) $ be a hypercomplex manifold. A hyperhermitian metric $ \Omega $ is \emph{Chern-Einstein} if the following is satisfied:
\begin{equation}\label{eq:Einstein}
\partial_J\alpha_{\Omega}=\lambda \Omega
\end{equation}
for some  $ \lambda \in C^{\infty}( M,  \R )$. The function $\lambda$ will be called \emph{Einstein factor}. If moreover $\Omega$ is HKT we say that it is \emph{HKT-Einstein}.
\end{defn}

\noindent We observe that, tracing \eqref{eq:Einstein},  we have
\[
s^{\mathrm{Ch}}(\Omega)=2n\lambda\,.
\]
With the next lemma we rephrase the definition in terms of the forms $ \omega_L $, for $ L\in \mathsf{H} $.

\begin{lem}\label{lem:einstein1}
Let $ (M,\mathsf{H}, g) $ be a hyperhermitian manifold. Then,  the following are equivalent:
\begin{enumerate}
\item \label{einst1} $ \Omega $ is Chern-Einstein with Einstein factor $ \lambda $;
\item \label{einst3} $\omega_I$ satisfies
\begin{equation*}
\frac{\Ric_{\omega_I}^{\mathrm{Ch}}-J\Ric_{\omega_I}^{\mathrm{Ch}}}{2}=\lambda \omega_I\,;
\end{equation*}
\item \label{einst2} For any pair of anticommuting complex structures $ L,P\in \mathsf{H} $, 
\[
\frac{\Ric_{\omega_L}^{\mathrm{Ch}}-P\Ric_{\omega_L}^{\mathrm{Ch}}}{2}=\lambda \omega_L\,.
\]
\end{enumerate}
\end{lem}
\begin{proof}
The equivalence of \eqref{einst1} and \eqref{einst3} follows from  Proposition \ref{Prop:alpha} (\ref{Lem:ricdelJalpha}). The fact that \eqref{einst1} is equivalent to \eqref{einst2} can be obtained using the  argument as in the proof of Proposition \ref{Prop:scalar}, thanks to which it is sufficient to assume $P=I$ and $L=aJ+bK\in \mathsf H$. On the other hand, this case follows from \eqref{eq:omegaL} and \eqref{eqn:Prop4.2}.
\end{proof}

 \noindent Using \eqref{eq:Ric_J} and \eqref{eq:Ric_K}, one can see that  $ \Omega $ is Chern-Einstein if and only if
\[
\left( \frac{\Ric_{\omega_J}^{\mathrm{Ch}}+\sqrt{-1} \Ric_{\omega_K}^{\mathrm{Ch}}}{2} \right)^{(2,0)}=\lambda \Omega\,.
\] We are grateful to M. Lejmi for this observation.

\subsection{The 1-dimensional case}\label{Subsec:1dim}
We start searching for examples of Chern-Einstein metric on compact hyperhermitian manifolds of quaternionic dimension $1$. These are always HKT for dimensional reasons and they have been classified by Boyer in \cite{Boyer}. Up to conformal equivalence,  the complete list is the following:
\begin{itemize}
    \item Tori with the flat metric;
    \item K3 surfaces with a hyperk\"ahler metric;
    \item Quaternionic Hopf surfaces with the standard locally conformally flat metric.
\end{itemize}
The first two classes are hyperk\"ahler while quaternionic Hopf surfaces are HKT non-K\"ahler, therefore they are the right candidate to check the HKT-Einstein condition. By  \cite{Kato}, a quaternionic Hopf surface is $M:=(\C^2\setminus\{0\})/\Gamma$,  where $\Gamma=\langle \gamma \rangle\ltimes G$ is conjugated to a subgroup in ${\rm SU}(2)\times \R^*$ and
$$
\gamma(z^1, z^2)=(a z^1, bz^2 )\,, \quad a, b\in \C\,, \quad 1<|b|\le |a|\,.
$$
Up to choosing suitable coordinates $\{z^1, z^2\}$ on $\C^2\setminus \{0\}$, we can assume that $G\subseteq {\rm SU}(2)$,  $|a|=|b|>1$, $ab\in \R$ and $Jdz^1=-d\bar z^2$. Hence, the SKT metric on $\C^2\setminus\{0\}$
\[
\omega_I=\frac{\sqrt{-1}}{|z|^2}  \left( dz^1\wedge d\bar z^1+ dz^2\wedge d\bar z^2 \right)\,,
\] descends to a  hyperhermitian metric on $M$. We claim that it is also HKT-Einstein with Einstein factor $\lambda=1$. We have
\[
\begin{aligned}
\Ric^{\mathrm{Ch}}_{\omega_I}&=2\omega_I-\frac{2\sqrt{-1}}{|z|^4}(\bar z^1 dz^1+\bar z^2dz^2)\wedge(z^1 d\bar z^1+z^2 d\bar z^2)\\
\end{aligned}
\]
concluding
\[
\frac{\Ric^{\mathrm{Ch}}_{\omega_I}-J\Ric^{\mathrm{Ch}}_{\omega_I}}{2}=\omega_I\,.
\] 

From this perspective the classification result of Boyer can be regarded as a hypercomplex version of the classical Uniformization Theorem. The fact that there is no representative with negative Einstein constant is not an accident. As a matter of fact we will prove in Proposition \ref{Prop:HKTEinstein} that this peculiar fact happens in any dimension for compact HKT-Einstein manifolds.

\subsection{The case \texorpdfstring{$ \lambda \not \equiv 0 $}{lambda not identically zero}}
Let $ (M^n,\mathsf{H}) $ be a hypercomplex manifold. If $c_1^{\mathrm{qBC}}(M,\mathsf H)$ has a sign, it admits a representative $\pm \Omega$ where $\Omega$ is an HKT metric.
Now, by definition,  there exists $ f\in C^\infty(M,\R) $ such that
\[
\partial_J \alpha_\Omega=\pm \Omega + \partial \partial_J f\,.
\]
It is then easy to show that the conformal change $\Omega_f = \mathrm{e}^{\frac{f}{n}}\Omega$ is hyperhermitian Einstein. Indeed
\[
\partial_J \alpha_{\Omega_f}=\partial_J \alpha_\Omega- \partial \partial_J f=\pm \Omega=\pm \mathrm{e}^{-\frac{f}{n}} \Omega_f\,.
\]
We shall however observe that the condition $ c_1^{\mathrm{qBC}}(M,\mathsf{H})<0 $ can never occur in the compact setting. Indeed, this would give the contradiction
\[
\begin{split}
0>-\int_M\Omega^n \wedge \bar \Omega_f^n&=\int_M \partial_J \alpha_{\Omega_f}\wedge \Omega^{n-1} \wedge \bar \Omega_f^n=\int_M \alpha_{\Omega_f}\wedge \Omega^{n-1} \wedge \partial_J \bar \Omega_f^n=\frac1n\int_M |\alpha_{\Omega_f}|^2_{\Omega} \Omega^n \wedge \bar \Omega_f^n\,.
\end{split}
\]
It is however important to observe that this behaviour is exclusive of the compact case, see Subsection \ref{subeinsteinnoncpt}.

Now, we can prove the following.
\begin{prop}
Let $(M^n,\mathsf{H},g)$ be a compact hyperhermitian manifold with $n\geq 2$ and admitting a quaternionic Gauduchon metric $\tilde{\Omega}$ conformal to $\Omega$. If
	\[
	\partial_J \alpha = \lambda \Omega
	\]
	for some $\lambda \in C^{\infty}(M, \R)$, $\lambda \not \equiv 0$, then $\tilde \Omega$ is  HKT and  $c_1^{{\rm qBC}}(M,\mathsf{H})>0$.
\end{prop}
\begin{proof}
Let  $f\in C^{\infty}(M, \R)$ be such that $\tilde \Omega=\mathrm{e}^f\Omega$. Applying $\partial_J$ to the Einstein equation, we obtain that
\begin{equation}\label{eq:Einstconf}
\partial_J (\lambda \mathrm{e}^{-f}) \wedge \tilde \Omega+\lambda \mathrm{e}^{-f}\partial_J \tilde \Omega=0\,.
\end{equation}
Let us now set $\psi=(\lambda \mathrm{e}^{-f})^{n-1}$. Then,  \eqref{eq:Einstconf} implies $\partial_J \psi \wedge \tilde \Omega^{n-1} +\psi \partial_J \tilde \Omega^{n-1}=0$. Applying $\partial$ to this identity,  we get
\[
\partial \partial_J \psi \wedge \tilde \Omega^{n-1}-\partial_J \psi \wedge \partial \tilde \Omega^{n-1}+\partial \psi \wedge \partial_J\tilde \Omega^{n-1}+\psi \partial \partial_J \tilde \Omega^{n-1}=0\,.
\]
Using the quaternionic Gauduchon condition and the definition of $\beta_{\tilde \Omega}$,  we obtain
\[
\Delta_{\tilde \Omega} \psi + \tilde g(d\psi, \beta_{\tilde \Omega}+\bar \beta_{\tilde \Omega})=0\,.
\]
Consequently, by the maximum principle, $\psi $ must be constant. Now also $\lambda \mathrm{e}^{-f}$ is a non-zero constant and \eqref{eq:Einstconf} reveals that $\tilde \Omega $ is HKT. Finally,  $\lambda$ cannot be negative otherwise we would have $c_1^{{\rm qBC}}(M, \mathsf H)<0$, which is impossible because $M$ is compact.
\end{proof}

\subsection{The HKT case}
   In view of  Subsection \ref{Subsec:1dim}, we will focus on dimension $n\ge2$. In this case, we can show the following.

\begin{prop}\label{Prop:HKTEinstein}
Let $ (M^n,\mathsf{H}, g) $ be a compact HKT-Einstein manifold with $n\ge 2$. Then, $s^{{\rm Ch}}$ is a non negative constant and equality holds if only if   $(M, \mathsf H)$ is ${\rm SL}(n, \H).$
\end{prop}
\begin{proof}
Let $\lambda$ be the Einstein factor of $\Omega$. Since $ \alpha $ and $ \Omega $ are $ \partial $-closed we have
\[
0=\partial \partial_J \alpha=\partial \lambda \wedge \Omega
\] 
therefore, for $ n\geq 2 $
\[
0=\partial \lambda \wedge \partial_J \lambda \wedge \Omega^{n-1}=\frac{1}{n}|\partial \lambda |^2 \Omega^n
\]
showing that $ \lambda $ is constant. Finally, since $c_1^{{\rm qBC}}(M, \mathsf{H})$ cannot be negative, then $\lambda\ge 0$. Furthermore, we have that
\[
\lambda= \frac{1}{\int_M \frac{\Omega^{n}\wedge \bar \Omega^{n}}{(n!)^2}}\int_M \lambda \frac{\Omega^{n}\wedge \bar \Omega^{n}}{(n!)^2}=\frac{1}{\int_M \frac{\Omega^{n}\wedge \bar \Omega^{n}}{(n!)^2}}\int_M \partial_J \alpha \wedge \frac{\Omega^{n-1}\wedge \bar \Omega^{n}}{(n!)^2}=\frac{1}{n}\frac{1}{\int_M \frac{\Omega^{n}\wedge \bar \Omega^{n}}{(n!)^2}}\int_M|\alpha |^2  \frac{\Omega^n\wedge \bar  \Omega^n}{(n!)^2}\,,
\]
so, clearly, $\lambda=0$ implies the $\mathrm{SL}(n,\H)$ condition. Conversely, if $(M,\mathsf{H}) $ is $\mathrm{SL}(n,\H)$ and $\lambda \neq 0$ then $\Omega$ would be globally $\partial \partial_J$-exact reaching a contradiction.
\end{proof}

\noindent It is fairly easy to come up with examples of compact hyperhermitian Einstein manifolds with $\lambda =0$ but not admitting compatible HKT metrics.  For instance,  any hypercomplex nilmanifold $M$ with a non-abelian hypercomplex structure $\mathsf{H}$,  by \cite[Theorem 4.6]{BDV},  does not have any compatible HKT metric. However,  since hypercomplex nilmanifolds are $ \mathrm{SL}(n,\H) $, any left-invariant hyperhermitian metric $ \Omega $ satisfies $\alpha=0$.

\medskip
Within our framework, the quaternionic Calabi conjecture formulated by Alesker and Verbitsky in \cite[Conjecture 1.5]{Alesker-Verbitsky (2010)} can be phrased as follows: let  $ (M,\mathsf{H},g) $ be a compact HKT manifold, for any  $ \Psi\in c_1^{\mathrm{qBC}}(M,\mathsf{H}) $, we can find $ \tilde{\Omega}\in[ \Omega]_{{\rm qBC}}$ such that $ \partial_J\alpha_{\tilde{\Omega}}=\Psi $. If this turns out to be true, when $ c_1^{\mathrm{qBC}}(M,\mathsf{H})=0 $ we would be able to find HKT-Einstein metrics with vanishing Einstein factor (cf. \cite{Verbitsky (2009)}), i.e. that are balanced and Chern-Ricci flat.\\
In a similar fashion, one is led to speculate on the existence of HKT-Einstein metrics on compact HKT manifolds with positive first quaternionic Bott-Chern class. More precisely, we wonder: 
\begin{quest}
Let $ (M^n,\mathsf{H}) $ be a compact hypercomplex manifold  such that $ c_1^{\mathrm{qBC}}(M,\mathsf{H})>0 $.  Does it always exist an HKT-Einstein metric $\tilde{\Omega}\in c_1^{{\rm qBC}}(M, \mathsf H) $? 
\end{quest}

\noindent Proceeding similarly to the K\"ahler case, the question turns out to be equivalent to the solvability of the following quaternionic Monge-Amp\`{e}re equation: let $\Omega\in c_1^{{\rm qBC}}(M, \mathsf H)$ be an HKT metric, 
\[
\left( \Omega+ \partial \partial_J \phi \right)^n=\mathrm{e}^{f-\phi} \Omega^n\,, \quad \Omega+ \partial \partial_J \phi>0\,,
\]
where $ \phi \in C^\infty(M,\R) $ is the unknown and $ f\in C^\infty(M,\R) $ is the datum. Unfortunately, this case is the quaternionic analogue of the Fano case in the K\"ahler setting, hence, when approaching the problem with the classical method of continuity, the same difficulties arise. Furthermore, it is natural to expect that certain obstructions may emerge in our context, similar to those found by Futaki \cite{Futaki} and Matsushima \cite{Matsushima}. It is extremely likely that more sophisticated tools are necessary to approach this problem. 

\subsection{Proof of Theorem \ref{Thm:main1.5}}\label{Subsec:Joyce}
In order to present the proof of  Theorem \ref{Thm:main1.5}, it will be useful to briefly overview the construction of  Joyce \cite{Joyce}. Let $G$ be a compact semisimple Lie group of rank $ r $ and fix a maximal torus $ H $ in $G$. Within this framework, structure theory can be performed, which allows to obtain a decomposition of  $ \mathfrak{g}={\rm Lie}(G) $  of the following form:
\begin{equation}\label{Joycedec}
\mathfrak{g}=\mathfrak{b}\oplus \bigoplus_{j=1}^m\mathfrak{d}_j\oplus \bigoplus_{j=1}^m\mathfrak{f}_j\,,
\end{equation}
where $ \mathfrak{b} $ is abelian of dimension $ r-m $, $ \mathfrak{d}_j\subseteq  \mathfrak{g} $ are  subalgebras isomorphic to $ \mathfrak{su}(2) $, and $ \mathfrak{f}_j\subseteq \mathfrak{g} $ are subspaces satisfying the following properties:
\begin{enumerate}[label=(J\arabic*),ref=J\arabic*]
\item\label{joyce1} $ [\mathfrak{d}_j,\mathfrak{b}]=0 $, for any $ j=1,\dots,m $,  and $ \mathfrak h :={\rm Lie}(H)\subseteq\mathfrak{b}\oplus \bigoplus_{j=1}^m\mathfrak{d}_j $;
\item\label{joyce2} $ [\mathfrak{d}_j,\mathfrak{d}_i]=0 $,  for $ j\neq i $;
\item\label{joyce3} $ [\mathfrak{d}_j,\mathfrak{f}_i]=0 $,  for $ j<i; $
\item\label{joyce4} $ [\mathfrak{d}_j,\mathfrak{f}_j]\subseteq \mathfrak{f}_j $, for any $ j=1,\dots,m $,  and this Lie bracket action is isomorphic to the direct sum of a finite amount of copies of the $ \mathfrak{su}(2) $-action on $ \C^2 $ by left matrix multiplication.
\end{enumerate}
Such a decomposition will be called a \emph{Joyce decomposition} of the Lie algebra $ \mathfrak{g}. $\\
Now, denote $ T^{2m-r}$ the $ (2m-r) $-dimensional torus, so that the Lie algebra of $ T^{2m-r}\times G $ decomposes as
\[
\tilde{\mathfrak g}:=(2m-r)\mathfrak{u}(1)\oplus\mathfrak{g} \simeq \R^m \oplus \bigoplus_{j=1}^m\mathfrak{d}_j\oplus \bigoplus_{j=1}^m\mathfrak{f}_j\,.
\] 
We define a hypercomplex structure $ I,J,K$ on $\tilde{\mathfrak{g}} $ in the following manner. For every $ j=1,\dots,m $, choose $ (e_1^j,e_2^j,e_3^j,e_4^j) $ a basis of $ \R \oplus \mathfrak{d}_j $ such that $ (e_1^1,e_1^2,\dots,e_1^m) $ is the standard basis of $ \R^m $ and $ e_2^j,e_3^j,e_4^j $ satisfy the commutation relations:
\begin{align}
\label{eq_basis_su2}
[e_2^j,e_3^j]=2e^j_4\,,&& [e_4^j,e^j_2]=2e^j_3\,,&& [e^j_3,e^j_4]=2e^j_2\,.
\end{align}
\begin{enumerate}
\item[(a)] For any $ j=1,\dots,m $, let $ I,J, K $ act on $ \R\oplus \mathfrak{d}_j $ as:
\begin{align*}
\quad \,\,\,  Ie^j_1=e^j_2\,, \quad Ie^j_3=e^j_4\,, \quad  Je^j_1=e^j_3\,, \quad Je^j_2=-e^j_4\,,
 \quad Ke^j_1=e^j_4\,, \quad Ke^j_2=e^j_3\,.
\end{align*}
We obviously further require $I^2=J^2=-\mathrm{Id}$.
\item[(b)] For any $ j=1,\dots,m $, let $ I,J, K$ act on $ \mathfrak{f}_j $ as:
\[
If=[e^j_2,f]\,, \quad Jf=[e^j_3,f]\,,
 \quad Kf=[e^j_4,f]\,,\quad f\in \mathfrak f_j\,.
\]
\end{enumerate}
Using \eqref{joyce4} and an argument due to Samelson \cite{Samelson (1953)}, one can see that   $ I,J,K$ induce a hypercomplex structure on $\tilde{ \mathfrak{g}}$, see \cite{Joyce}. Observe that different choices of the isomorphism $(2m-r)\mathfrak u(1)\oplus \mathfrak b\simeq \R^m$ lead to, possibly, non-equivalent hypercomplex structures on $\tilde{\mathfrak g}$, see \cite{Opfermann-Papadopoulos}.
 It was shown in \cite{Grantcharov-Poon (2000), Opfermann-Papadopoulos} that, for any group $G$, there exists at least a hypercomplex structure for which the opposite of the Cartan-Killing form $B$ on $G$ can be extended to  a hyperhermitian metric on $ T^{2m-r}\times G $, which, in turns, is strong HKT. On the other hand, we shall remark that, for some choices of $G$, there exists a family of Joyce's hypercomplex structures which are not compatible with any bi-invariant metric on  $T^{2m-r}\times G$, see \cite[p. 16]{Opfermann-Papadopoulos}. Furthermore, it is not even known whether invariant HKT metrics compatible with such hypercomplex structure exist.

We shall now  prove Theorem \ref{Thm:main1.5}. The latter provides a rather large class of positive HKT-Einstein manifolds and, as a by-product, proves the existence of invariant HKT metrics compatible with any Joyce hypercomplex structure.

\begin{proof}[Proof of Theorem \ref{Thm:main1.5}]
Let \eqref{Joycedec}
be a Joyce decomposition of $ \mathfrak{g} $ and let $g:=g_E \oplus (-B)$, where $g_E$ is the Euclidean metric on $(2m-r) \mathfrak{u}(1)$. Then, $g$ is bi-invariant on $\tilde{ \mathfrak{g}}$ and, even though $g$ is not hyperhermitian, the Joyce decomposition is $g$-orthogonal \cite[Lemma 2]{GP}. Denote $ d_j=\dim_\H(\mathfrak{f}_j) $. We choose a $g$-orthogonal basis $\mathcal B:=(e_1^1,\dots,e_4^1,\ldots, e_1^m,\dots,e_4^m,f^1_1,\dots,f^1_{4d_1},\ldots, f^m_1,\dots,f^m_{4d_m})$ of $\tilde{ \mathfrak{g}}$ such that: $ (e_{1}^j,\dots,e^j_4) $ is an orthogonal basis of $ \R\oplus \mathfrak{d}_j $ for any $ j=1,\dots,m $ with only non-zero commutators as in \eqref{eq_basis_su2} while $\{ f^j_1,\dots,f^j_{4d_j}\} $ is  any orthonormal basis of $\mathfrak f_j$ such that, for $ k=1,\dots,d_j $,
$$ f^j_{4k-2}=If^j_{4k-3}\,,\quad  f^j_{4k-1}=Jf^j_{4k-3}\,, \quad  f^j_{4k}=Kf^j_{4k-3} \,.
$$ This last choice is possible since $g$ is hyperhermitian on $\mathfrak f_j$, for any $j=1, \ldots, m$.  We will also denote $(e_1,\dots,e_{4n})$ the $g$-orthonormal basis obtained from $\mathcal B$ by rescaling the vectors to unit $g$-norms. 
By \cite[Proposition 4.1]{ Vezzoni}, the form
$
\Ric^{\mathrm{Ch}}_{\omega_I}-J\Ric^{\mathrm{Ch}}_{\omega_I}$
is independent of the left-invariant hyperhermitian metric $\omega_I$. 
Then, in view of Lemma \ref{lem:einstein1} (\ref{einst3}), it is enough to show that
$$
\frac{\Ric^{\mathrm{Ch}}_{\omega_I}-J\Ric^{\mathrm{Ch}}_{\omega_I}}{2}(X,IX)>0\,,  \quad X\in \tilde{ \mathfrak{g}}\,, \quad X\ne 0 \,.
$$
Using  \cite[Proposition 4.1]{ Vezzoni}, we have that 
\begin{equation}\label{eq:Jric}
\begin{aligned}
\frac{\Ric^{\mathrm{Ch}}_{\omega_I}-J\Ric^{\mathrm{Ch}}_{\omega_I}}{2}(X,IX)=&\, -\frac{1}{4}\mathrm{tr}(\mathrm{ad}_{[X,IX]}I)-\frac{1}{4}\mathrm{tr}(I\mathrm{ad}_{[JX,KX]})\\
=&\, \frac14\sum_{i=1}^{4n} g([X,IX]+[JX,KX],[e_i, Ie_i])\,.
\end{aligned}
\end{equation}

Thus, we need to compute $\sum_{i=1}^{4n}[e_i,Ie_i]$.  By  \eqref{eq_basis_su2} and bi-invariance, for all $j=1,\dots,m$, we  can see that $\lambda_j=:\| e_2^j\|_g=\|e_3^j\|_g=\|e_4^j\|_g$. Hence,  we obtain 
\[
\sum_{i=1}^{4m}[e_i,Ie_i]=\sum_{j=1}^m\left(\frac{1}{\|e_1^j\|^2_g}[e_1^j,Ie_1^j]+\frac{1}{\lambda_j^2}[e_2^j,Ie_2^j]+\frac{1}{\lambda_j^2}[e_3^j,Ie_3^j]+\frac{1}{\lambda_j^2}[e_4^j,Ie_4^j]\right)=\sum_{j=1}^m\frac{4}{\lambda_j^2} e_2^j\,.
\]
It remains to compute $\sum_{i=4m+1}^{4n}[e_i,Ie_i]=\sum_{j=1}^m \sum_{k=1}^{4d_j}[f^j_k,If^j_k]$. To do this, we will first prove the claim that, for any $ j=1,\dots,m $, 
\begin{equation}\label{eq:in_Joyce}
\sum_{i=1}^{4d_j} [f_i^j,If_i^j]=\frac{4d_j}{\lambda_j^2} e^j_2\,.
\end{equation}
Indeed, we observe that
\[
\sum_{r=0}^3[f^j_{4k-r},If^j_{4k-r}]=2\left([f^j_{4k-3},f^j_{4k-2}]+[f^j_{4k-1},f^j_{4k}]  \right).
\]
By definition of $ (I,J,K) $ on $ \mathfrak{f}_j $ and using Jacobi's identity, we have
\[
\begin{split}
[f^j_{4k-1},f^j_{4k}]&=[f^j_{4k-1},Kf^j_{4k-3}]=[[f^j_{4k-1},e^j_4],f^j_{4k-3}]+[e^j_4,[f^j_{4k-1},f^j_{4k-3}]]\\
&
=-[Kf^j_{4k-1},f^j_{4k-3}]+[e^j_4,[f^j_{4k-1},f^j_{4k-3}]]
=[f^j_{4k-2},f^j_{4k-3}]+[e^j_4,[f^j_{4k-1},f^j_{4k-3}]]\,,
\end{split}
\]
hence
\[
\sum_{i=1}^{4d_j} [f_i^j,If_i^j]=2\sum_{k=1}^{d_j} [e^j_4,[f^j_{4k-1},f^j_{4k-3}]]\,.
\]
Now,  we fix $ k\in\{1, \ldots, d_j\} $ and consider the components of $ [f^j_{4k-1},f^j_{4k-3}] $ with respect to Joyce's decomposition:
\[
[f^j_{4k-1},f^j_{4k-3}]=\sum_{p=1}^m \left( D_1^pe^p_1+D^p_2e^p_2+D^p_3e^p_3+D^p_4e^p_4+\sum_{q=1}^{4d_p}F^p_qf^p_q\ \right) .
\]
Then, by, respectively,  \eqref{joyce1},  \eqref{joyce2} and \eqref{joyce3},  we have that 
 $ [e^j_4,e^p_1]=0 $, for $ p=1,\dots,m $,   $ [e^j_4,e^p_q]=0 $, for $ p\neq j $ and $ q\in \{2,3,4 \} $ and $ [e^j_4,f^p_q]=0 $, for $ p>j $ and $ q\in \{1,\dots,4d_p\} $.  Moreover, we have
\[
\begin{split}
D^j_2&=\frac{1}{\|e_2^j\|^2_g}g([f^j_{4k-1},f^j_{4k-3}],e^j_2)=\frac{1}{\lambda_j^2}g(f^j_{4k-1},[f^j_{4k-3},e^j_2])
=-\frac{1}{\lambda_j^2}g(f^j_{4k-1},If^j_{4k-3})=0\,,
\end{split}
\]
while, proceeding similarly, we have  $D^j_3=-\frac{1}{\lambda_j^2}$ and $D_4^j=0$.
Furthermore,  for $ p<j $ and $ q\in \{1,\dots,4d_p\} $, we have 
\[
\begin{split}
F^p_q&=- g([f^j_{4k-2},f^j_{4k-3}],I^2f^p_q)
=-  g([[f^j_{4k-2},f^j_{4k-3}],e^p_2],If^p_q)
\\
&=- g([[f^j_{4k-2},e^p_2],f^j_{4k-3}]+[f^j_{4k-2},[f^j_{4k-3},e^p_2]],If^p_q)=0\,,
\end{split}
\]
where in the last equality we used  \eqref{joyce3}.  For $p=j$, as above, we have that 
\[
\begin{split}
F^j_q&=-g([[f^j_{4k-2},e^j_2],f^j_{4k-3}]+[f^j_{4k-2},[f^j_{4k-3},e^j_2]],If^j_q)\\
&=g([If^j_{4k-2},f^j_{4k-3}]+[f^j_{4k-2},If^j_{4k-3}],If^j_q)\\
&=g(-[f^j_{4k-3},f^j_{4k-3}]+[f^j_{4k-2},f^j_{4k-2}],If^j_q)=0\,.
\end{split}
\]
Then, putting everything together we infer
\[
[e^j_4,[f^j_{4k-1},f^j_{4k-3}]]=- \frac{1}{\lambda_j^2}[e^j_4,e^j_3]=\frac{2}{\lambda_j^2} e^j_2
\]
and summing up over $ k $ we obtain \eqref{eq:in_Joyce}, as claimed. Therefore
\[
\sum_{i=1}^{4n}[e_i,Ie_i]=4  \sum_{j=1}^m \frac{1+d_j}{\lambda_j^2} e^j_2\,,
\]
and, hence, \eqref{eq:Jric} becomes
$$
\frac{\Ric^{\mathrm{Ch}}_{\omega_I}-J\Ric^{\mathrm{Ch}}_{\omega_I}}{2}(X,IX)=\sum_{j=1}^m \frac{1+d_j}{\lambda_j^2} g([X,IX]+[JX,KX],e_2^j)\,.
$$
 We will prove that $g([X,IX]+[JX,KX],e_2^j)\ge 0$, for any $X\in \tilde{ \mathfrak{g}}$. To do so,   we rewrite it as
\begin{equation}\label{eqn:compoofJ-invric}
\begin{split}
g([X,IX]+[JX,KX],e_2^j)&=g(IX,[e_2^j,X])+g(KX, [e_2^j,JX])
\\
&=\sum_{i=1}^{4n}g(X,e_i)g(IX,[e_2^j,e_i])+\sum_{i=1}^{4n}g(JX,e_i)g(KX,[e_2^j,e_i])\,.
\end{split}
\end{equation}
If $e_i\in (2m-r)\mathfrak{u}(1)\oplus \mathfrak{b} \oplus \bigoplus_{k\neq j} \mathfrak{d}_k \oplus \langle e_2^j \rangle \oplus \bigoplus_{k>j}\mathfrak{f}_k $,  then $[e_2^j,e_i]=0$. On the other hand, if $e_i=\frac{e_3^j}{\lambda_j}$, then, using that $g$ is $I$-Hermitian along $e_3^j$ and $e_4^j$, 
\[
\begin{split}
g(X,e_i)g(IX,[e_2^j,e_i])+g(JX,e_i)g(KX,[e_2^j,e_i])&=
\frac{2}{\lambda_j^2}\left( g(X,e_3^j)^2+g(JX,e_3^j)^2\right)\,,
\end{split}
\]
Similarly, for $e_i=\frac{e_4^j}{\lambda_j}$, we obtain
\[
g(X,e_i)g(IX,[e_2^j,e_i])+g(JX,e_i)g(KX,[e_2^j,e_i])=\frac{2}{\lambda_j^2}\left( g(X,e_4^j)^2+g(JX,e_4^j)^2\right)\,.
\]
If $e_i\in \mathfrak{f}_k$ with $k<j$ then $[e_2^j,e_i],[e_2^j,Je_i]\in \mathfrak{f}_k$, see \cite[Lemma 3.1]{BFGV2}, and thus, denoting $X_{\mathfrak{f}_k}$ the component of $X$ along $\mathfrak{f}_k$, we get
\[
\begin{split}
\sum_{e_i\in\mathfrak{f}_k}g(X,e_i)&g(IX,[e_2^j,e_i])+g(JX,e_i)g(KX,[e_2^j,e_i])\\
& =\sum_{e_i\in\mathfrak{f}_k}g(X_{\mathfrak{f}_k},[e_3^k,Je_i])g(IX_{\mathfrak{f}_k},[e_2^j,[e_3^k,Je_i]])-g(X_{\mathfrak{f}_k},[e_3^k,e_i])g(IX_{\mathfrak{f}_k},[e_3^k,[e_2^j,e_i]])\\
& =\sum_{e_i\in\mathfrak{f}_k}g(X_{\mathfrak{f}_k},[e_3^k,Je_i])g(IX_{\mathfrak{f}_k},[e_2^j,[e_3^k,Je_i]])-g(X_{\mathfrak{f}_k},[e_3^k,e_i])g(IX_{\mathfrak{f}_k},[e_2^j,[e_3^k,e_i]])=0\,,
\end{split}
\]
where the last equality is due to Jacobi's identity. Finally, if $e_i\in \mathfrak{f}_j$,  we have
\[
\begin{split}
g(X,e_i)&g(IX,[e_2^j,e_i])+g(JX,e_i)g(KX,[e_2^j,e_i])
=g(X,e_i)^2+g(X,Je_i)^2\,,
\end{split}
\] thanks to the fact that $g|_{\mathfrak f_j}$ is hyperhermitian.
Putting all these computations together, we have  that \eqref{eqn:compoofJ-invric} rewrites as
\[
\begin{split}
g([X,IX]+[JX,KX],e_2^j)&=\frac{2}{\lambda_j^2}\sum_{l=3,4}g(X,e_l^j)^2+g(JX,e_l^j)^2
+2\sum_{k=1}^{d_j}\sum_{r=0}^3g(X,f_{4k-r}^j)^2\ge0\,,\\
\end{split}
\] for any $X\in\tilde{ \mathfrak{g}}$.  Now, summing over $j=1,\ldots, m$,  after noticing that $\frac{\Ric^{\mathrm{Ch}}_{\omega_I}-J\Ric^{\mathrm{Ch}}_{\omega_I}}{2}(X,IX)=0$ if and only if $X=0$, we conclude the proof.
\end{proof}

\section{Examples and constructions}\label{Sec:Ex}

In this section,  we collect several examples and two interesting constructions, one inspired by Arroyo and Nicolini \cite[Section 5]{AN}, the other due to Barberis and Fino \cite{Barberis-Fino}. The purpose of the  section is to exhibit examples admitting one type of  metrics listed in Definition \ref{def:specialmetrics} but none of the one immediately stronger. Finally, Subsection \ref{subeinsteinnoncpt} is devoted to present some non-compact HKT-Einstein manifolds, two of which have negative Einstein factor.

\smallskip
In the spirit of searching special metrics, it will be useful to note that the well-known symmetrisation technique on a compact quotient $M$ of a simply connected Lie group $G$ by a lattice $\Gamma$ endowed with a left-invariant hypercomplex structure $\mathsf H$ allows to consider invariant ones. 
Indeed, let $ \Omega $ be a hyperhermitian metric on $ (M,\mathsf{H}) $. Since the symmetrisation map $ \mu $ commutes with $ I,J $ and $ d $ it also commutes with $ \partial $ and $ \partial_J $. Moreover, if  $ \mu(\Omega^{n-1}) $ is q-positive, then by Lemma \ref{Lem:bijection} there exists an invariant hyperhermitian metric $ \tilde \Omega $ such that $ \tilde \Omega^{n-1}=\mu(\Omega^{n-1}) $. From this point it is straightforward to conclude that if $ (M,\mathsf{H}) $ admits a metric which is quaternionic balanced, quaternionic strongly Gauduchon or quaternionic Gauduchon then it admits an invariant one. Moreover, using the symmetrisation, from \cite[Proposition 16]{GLV} and Lemma \ref{Lem:lemma6.1} we deduce:

\begin{lem}
Let $(M^n,\mathsf{H})$ be an $\mathrm{SL}(n,\H)$ compact quotient of a Lie group by a lattice such that the hypercomplex structure $ \mathsf{H} $ is left-invariant, then any invariant hyperhermitian metric is quaternionic Gauduchon, moreover it is quaternionic balanced if and only if it is balanced.
\end{lem}

Motivated by the facts above, in the next subsections, we will work on Lie algebras endowed with a hypercomplex structure. So, we will always fix a coframe $\{e^1,\ldots, e^{4n}\}$  for the Lie algebra and,  unless otherwise stated, use the following hypercomplex structure: for all $ k=1,\ldots, n $,
\[
Ie^{4k-3}=-e^{4k-2}\,, \quad Ie^{4k-1}=-e^{4k}\,, \quad Je^{4k-3}=-e^{4k-1}\,, \quad Je^{4k-2}=e^{4k}\,.
\]
In view of this, the $(1,0)$-coframe $\{\zeta^1,\ldots, \zeta^{2n}\}$ with respect to $I$ will be $\zeta^j=e^{2j-1}+\sqrt{-1}e^{2j}$, for all $j=1, \ldots, 2n$,  and $J$ will act on it as $J\zeta^{2i-1}=-\bar \zeta^{2i}$, for all  $i=1, \ldots, n$. As customary, we will denote $\zeta^{ij}:=\zeta^i\wedge \zeta^j$ and $\zeta^{i\bar j }:=\zeta^i\wedge \bar \zeta^j.$

\subsection{Quaternionic balanced manifolds without HKT metrics}\label{Ex:quat_bal}
An example of a quaternionic balanced manifold not admitting any HKT metric can be found, for instance, in \cite{Fino-Grantcharov-Verbitsky}. 
We provide another  $3$-dimensional example.

\begin{es}\label{ex:Lejmi-Weber}
We consider the hypercomplex nilpotent Lie algebra in \cite[Example 3]{Lejmi-Weber} with structure equations:
\[
d \zeta^i=0\,, \quad \mbox{for all }i=1,\ldots, 4\,, \quad d \zeta^5=\frac{1}{2}(\zeta^{13}+ \zeta^{\bar 13})\,, \quad d \zeta^6=\frac{1}{2}(\zeta^{14}+ \zeta^{\bar 14})\,.
\]
Since the hypercomplex structure is not abelian, by \cite[Theorem 4.6]{BDV} the nilmanifold $ N $ does not admit any HKT metric.  On the other hand the invariant hyperhermitian metric on $ N $ that makes the above coframe unitary is quaternionic balanced.\\
Observe that this example does not satisfy the $ \partial \partial_J $-Lemma because of \cite[Theorem 5]{Lejmi-Tardini} and the observation of Lejmi and Weber \cite[Example 3]{Lejmi-Weber} that it is not $ C^\infty $-pure nor $ C^\infty $-full.
\end{es}

\subsection{Quaternionic strongly Gauduchon manifolds without quaternionic balanced metrics} \label{qsGnoqb}
Notice that the minimum quaternionic dimension for which it is possible to provide such an example as a nilmanifold is $ 3 $, indeed all hypercomplex nilmanifolds are $ \mathrm{SL}(n,\H) $, see \cite[Corollary 3.3]{BDV}. Therefore, in real dimension $ 8 $, they admit a quaternionic strongly Gauduchon metric if and only if they admit an HKT metric by \cite[Theorem 10.1]{Lejmi-Weber}.

\begin{es}\label{es:qsGnoqB}
Consider the hypercomplex nilpotent Lie algebra with structure equations:
\[
d\zeta^i=0\,, \quad \text{for all }i=1,\dots,4\,,
\quad
d\zeta^5=\frac{1}{2}(\zeta^{12}+ \zeta^{\bar 12}-\zeta^{4\bar 4})\,, \quad d\zeta^6=\frac{1}{2}(\zeta^{34}+ \zeta^{\bar 34}+\zeta^{2\bar 2})\,.
\]
By straightforward computations, one can see that the invariant hyperhermitian metric $ \Omega $ that renders the coframe unitary is quaternionic strongly Gauduchon as
\[
\begin{aligned}
\partial \Omega^2&=2\partial_J(\zeta^{3456}-\zeta^{1256})\,.
\end{aligned}
\]
Since the structure constants are rational, the Lie algebra admits a lattice and the metric above descends to the compact quotient $ N $. On the other hand,  we claim that there are no quaternionic balanced metrics on the nilmanifold $ N $. Indeed, the form $ \zeta^{12}=2\partial \zeta^5 $ induces a q-positive, $ \partial $-exact, $ (2,2n) $-current $ T $ given by integration
\[
T(\gamma):=\int_N \zeta^{12} \wedge \gamma\wedge \bar \Theta \,, \quad \gamma\in \Lambda_I^{2n-2,0}N\,,
\]
where $ \Theta $ is any holomorphic volume form in $ \Lambda^{2n,0}_IN $. We conclude applying Proposition \ref{qPaulcurrent}.
\end{es}

\begin{rmk}
Both the quaternionic balanced condition and the strongly Gauduchon condition are not preserved under small deformations of the hypercomplex structure, see  \cite{Fino-Grantcharov}. 
It would be interesting to know if the same phenomenon occurs for the quaternionic Gauduchon condition.
\end{rmk}

\subsection{Quaternionic Gauduchon manifolds without quaternionic strongly Gauduchon metrics}
\label{Sec:qGnoqsG}
We already mentioned that any $ 8 $-dimensional hypercomplex nilmanifold with non-abelian hypercomplex structure admits no quaternionic strongly Gauduchon metric. On the other hand it admits a quaternionic Gauduchon metric by \cite[Proposition 16]{GLV}. Notice that, as a consequence of Lemma \ref{Lem:deldelJlemma}, no example in this class can satisfy the $ \partial \partial_J $-Lemma.\\
Here we present an example in each dimension.

\begin{es}
For $ n\geq 2 $,  consider the $ n $-dimensional hypercomplex nilpotent Lie algebra with structure equations:
\[
d \zeta^i=0\,, \quad i=1,\ldots, {2n-2}\,,
\quad
d \zeta^{2n-1}=-\frac{1}{2}\sum_{k=1}^{n-1} \zeta^{2k-1,  \overline{2k-1}}\,, \quad d \zeta^{2n}=\frac{1}{2}\sum_{k=1}^{n-1}(\zeta^{2k-1,  2k}+ \zeta^{\overline{2k-1}, 2k})\,.
\]
We shall check that the generic invariant hyperhermitian metric
\[
\Omega= A_{ij}\zeta^{ij}\,, \quad A_{ij}\in \C\,,
\]
is not quaternionic strongly Gauduchon. We easily see that
\[
\begin{split}
\partial \Omega^{n-1}
&=-\frac{(n-1)!}{2}\left(\sum_{k=1}^{n-1} \pf(A(2k-1,2k)) \right) \zeta^{1 \ldots 2n-1}\,,
\end{split}
\]
where  $ A(i,j) $ denotes the matrix obtained from the skew-symmetric matrix $ A=(A_{rs}) $ removing the $ i $\textsuperscript{th} and $ j $\textsuperscript{th} rows and columns. Then $ \partial \Omega^{n-1} $ is non-zero, because by q-positivity of $ \Omega $ we must have $ \pf(A(2k-1,2k)) >0 $, for all $ k=1,\dots,n $.  Moreover,  it is not $ \partial_J $-exact since
\[
\partial_J \zeta^i=0\,, \quad i\neq 2n-1\,, \quad \partial_J \zeta^{2n-1}=-\frac{1}{2}\sum_{k=1}^{n-1}\zeta^{2k-1,  2k}\,.
\]
On the other hand, these are all nilpotent Lie algebras with rational structure constants and any hypercomplex nilmanifold admits quaternionic Gauduchon metrics by  \cite[Proposition 16]{GLV}.
\end{es}

\subsection{A hypercomplex manifold without quaternionic Gauduchon metrics}
\label{Ex:Andrada-Tolcachier}

The example we describe appeared in \cite[Example 6.3]{AT}. 
It is constructed as a compact quotient $ S=\Gamma \backslash G $ of a solvable Lie group by a lattice $ \Gamma $ equipped with an invariant hypercomplex structure $ \mathsf{H} $. Andrada and Tolcachier show that the canonical bundle of $ (S,I) $ is holomorphically trivial but that of $ (S,J) $ is not, indeed, they even show that $ c_1^{\mathrm{BC}}(S,J)\neq 0 $. On the other hand, the holomorphic triviality of $ K(S,I) $ implies $ c_1^{\mathrm{qBC}}(S, \mathsf{H})=0 $ (cf. Remark \ref{rmk:implications}). We then conclude that a compatible quaternionic Gauduchon metric cannot exists from Theorem \ref{Thm:main1.2}.

\begin{rmk}\label{rmk:kodAT}
This example, together with Proposition \ref{Prop:kod} also shows that $ \kappa(S,I)=0 $ while $ \kappa(S,J)=-\infty $, so, in general, different complex structures in the same hypercomplex structure may yield different Kodaira dimensions. From this,  it also follows that the complex deformation
\[
I_t=\cos(\pi t)I + \sin(\pi t) J\,, \quad t\in [0,1/2]\,,
\]
does not preserve the Kodaira dimension, a fact that is in contrast with the projective case as shown by Siu \cite{Siu}. We are grateful to G. Grantcharov for this observation.
\end{rmk}

\subsection{Dependence of the quaternionic strongly Gauduchon condition from the complex structures}\label{Subsec:qsG}

Example \ref{es:qsGnoqB} can be used to show that the quaternionic strongly Gauduchon condition depends on the choice of the pair of anticommuting complex structures in the same hypercomplex structure. Indeed, let us show that with respect to the pair $ (J,I) $ there are no quaternionic strongly Gauduchon metrics. Consider the $ (1,0) $-coframe with respect to $ J $ given by $ w^{2k-1}=e^{4k-3}+\sqrt{-1}e^{4k-1} $ and $ w^{2k}=e^{4k-2}-\sqrt{-1}e^{4k} $ for $ k=1,2,3 $, and let $ d=\partial+\bar \partial $ be the splitting of the exterior differential with respect to $ J $. Then
\begin{align*}
\partial w^i&=0\,, \quad i=1,\dots,5\,, & \partial w^6&=\frac{1}{2}(\sqrt{-1}w^{12}+w^{34})\,,\\
I^{-1}\bar \partial I w^i&=0\,, \quad i\neq 5\,, & I^{-1}\bar \partial I w^5&=\frac{1}{2}(\sqrt{-1}w^{12}-w^{34})\,.
\end{align*}
It is now  easy to check that the generic invariant hyperhermitian metric $ \Omega $ cannot be quaternionic strongly Gauduchon, since $ \partial \Omega^2 $ is proportional to $ w^{12345}$ which is never $ I^{-1}\bar \partial I $-exact.

\subsection{Arroyo-Nicolini's construction}\label{Subsec:AN}
In \cite[Section 5]{AN}, Arroyo and Nicolini give a procedure to construct nilpotent SKT Lie algebras starting from previous ones. Here we adapt their argument to build new hyperhermitian nilpotent Lie algebras.

Let $ (\mathfrak{g}_1,[\cdot,\cdot]_1,\mathsf{H}_1) $ and $ (\mathfrak{g}_2,[\cdot,\cdot]_2,\mathsf{H}_2) $ be two nilpotent Lie algebras equipped with a hypercomplex structure generated by $ (I_1,J_1) $ and $ (I_2,J_2) $ respectively. Provided  we can choose $ e_i\in \mathfrak z(\mathfrak{g}_i)\setminus [\mathfrak{g}_i,\mathfrak{g}_i] $, where $ \mathfrak z(\mathfrak{g}_i) $ is the centre of $ \mathfrak{g}_i $. Then, we can define the new Lie algebra $ \mathfrak{g}=\mathfrak{g}_1 \oplus \mathfrak{g}_2 \oplus \langle X,Y,Z,W \rangle $ with Lie bracket
\[
[\cdot,\cdot]\vert_{\mathfrak{g}_i\times \mathfrak{g}_i}:=[\cdot,\cdot]_i\,, \quad [X,Y]=-[Z,W]:=e_1+e_2
\]
and hypercomplex structure $ \mathsf{H} $ determined by:
\begin{align*}
	& I\vert_{\mathfrak{g}_i}=I_i\,, && IX=Y\,, && IZ=W\,,\\
	& J\vert_{\mathfrak{g}_i}=J_i\,, && JX=Z\,, && JY=-W\,.
\end{align*}
We note that $ \mathfrak{g}_2 $ and  that, if $ \mathfrak{g}_1 $ and $ \mathfrak{g}_2 $ have rational structure constants so does $ \mathfrak{g} $, therefore we can find lattices.

\begin{thm}\label{Thm:AN}
The hypercomplex Lie algebra $ (\mathfrak{g},\mathsf{H}) $ admits HKT (resp. quaternionic balanced, quaternionic strongly Gauduchon) metrics if and only if $ (\mathfrak{g}_1,\mathsf{H}_1) $ and $ (\mathfrak{g}_2,\mathsf{H}_2) $ do.
\end{thm}
\begin{proof}
Let $ \Omega_i $ is a hyperhermitian metric on $ \mathfrak{g}_i $ for $ i=1,2 $. We define\begin{equation}\label{eq:AN}
\Omega:= \Omega_1+\Omega_2+\zeta^{12}\,,
\end{equation} 
where $ (\zeta^1,\zeta^2) $ is the $ (1,0) $ coframe with respect to $ I $ dual to $ (X-\sqrt{-1}Y,Z-\sqrt{-1}W) $. Clearly, $\Omega$ is a hyperhermitian metric on $\mathfrak g$.  By   straightforward computations, using that $ \partial\zeta^{12}=0 $, we obtain that $\Omega$ is,  respectively, HKT, quaternionic balanced or quaternionic strongly Gauduchon if and only if $\Omega_i$, $i=1,2$, are so.\end{proof}

\subsection{Barberis-Fino's construction}\label{Subsec:FB}
In this subsection,  we study the behaviour of the metric conditions under a construction due to Barberis and Fino \cite{Barberis-Fino}. The idea is to take a hypercomplex Lie algebra and build a new one via a quaternionic representation.

Let $ \mathfrak{g} $ be a $ 4n $-dimensional Lie algebra and $ \rho \colon \mathfrak{g} \to \mathfrak{gl}(k,\H) $ a Lie algebra homomorphism. Define on $ T_\rho \, \mathfrak{g}:=\mathfrak{g}  \ltimes_\rho \H^k $ the Lie bracket
\[
[(X,U),(Y,V)]:=([X,Y],\rho_X(V)-\rho_Y(U))
\]
for every $ X,Y\in \mathfrak{g} $ and $ U,V\in \H^k $. If $ \mathfrak{g} $ is endowed with a hypercomplex structure $ \mathsf{H} $
\[
\tilde L(X,U):=(LX,lU)\,,
\]
defines a hypercomplex structure $ \tilde {\mathsf{H}} $ on $ T_\rho \, \mathfrak{g} $, where if $ L=aI+bJ+cK\in \mathsf{H} $ then $ l=ai+bj+ck $, being $ i,j,k $ the quaternion units in $ \H $. Finally, given   a hyperhermitian metric $ \Omega $ on $ (\mathfrak{g},\mathsf{H}) $, the metric $ \tilde \Omega=\Omega\oplus \Omega_{{\rm std}} $ is hyperhermitian on $ (T_\rho \, \mathfrak{g},\tilde {\mathsf{H}})$, where $\Omega_{{\rm std}}$ is the  standard hyperhermitian metric on $ \H^k $. We will denote objects on $ (T_\rho \, \mathfrak{g},\tilde {\mathsf{H}}) $ with a tilde. In \cite[Theorem 3.2, Proposition 3.1]{Barberis-Fino} was shown that the properties of being hyperk\"ahler, HKT, strong HKT and balanced  are preserved under the above construction provided $\rho\colon \mathfrak g \to \mathfrak{sp}(k)$. We can refine these results by proving the following.



\begin{thm}
Let $ (\mathfrak{g},\mathsf{H},g) $ be a hyperhermitian Lie algebra and  $\rho \colon \mathfrak{g} \to \mathfrak{sp}(k) $. Denote $ p\colon T_\rho \, \mathfrak{g} \to \mathfrak{g} $ the orthogonal projection.  Then,  $ \alpha_{\tilde \Omega}=p^*\alpha_\Omega  $, $ \beta_{\tilde \Omega}=p^*\beta_\Omega$, $ \Ric^{\mathrm{Ch}}_{\omega_{\tilde L}}=p^*\Ric^{\mathrm{Ch}}_{\omega_L}$ and $ \Ric^{\mathrm{Bis}}_{\omega_{\tilde L}}=p^*\Ric^{\mathrm{Bis}}_{\omega_L} $. In particular,  $  \tilde \Omega $ is quaternionic balanced (resp. quaternionic Gauduchon) if and only if $ \Omega $ is.
\end{thm}
\begin{proof}
Using  the expression of $\tilde\Omega$, 
it follows that
\[
\begin{split}
d \tilde \Omega ((X,U),(Y,V),(Z,W))=&\,  d \Omega(X,Y,Z)-\Omega_{{\rm std}}((\rho_X+\rho_X^*)V,W)-\Omega_{{\rm std}}((\rho_Y+\rho_Y^*)W,U)\\
\,\, &-\Omega_{{\rm std}}((\rho_Z+\rho_Z^*)U,V)\,.
\end{split}
\]
Therefore, since  $ \rho \colon \mathfrak{g} \to \mathfrak{sp}(k) $, we have  $ d\tilde \Omega=p^*d\Omega $. Taking $ (3,0) $-parts and applying $ \Lambda $ yields
\begin{equation}\label{eqn:Thm9.8}
\beta_{\tilde \Omega}(X,U)=(\Lambda_{\tilde {\Omega}} \partial \tilde \Omega)(X,U)=(\Lambda_{\Omega}\partial \Omega) (X )=\beta_\Omega(X)\,.
\end{equation}
Now,  using Proposition \ref{Prop:alpha} \eqref{alphaLee},  \cite[Proof of Proposition 3.1 (iii)]{Barberis-Fino} and \eqref{eqn:Thm9.8}, we deduce the claim for $\alpha_{\tilde \Omega}$. This, in turns, implies $ \Ric^{\mathrm{Ch}}_{\omega_{\tilde L}}=p^*\Ric^{\mathrm{Ch}}_{\omega_L} $ as well as $ \Ric^{\mathrm{Bis}}_{\omega_{\tilde L}}=p^*\Ric^{\mathrm{Bis}}_{\omega_L}$.
\end{proof}


\subsection{Some non-compact HKT-Einstein manifolds}\label{subeinsteinnoncpt}
We now provide examples of invariant HKT-Einstein metrics on non-unimodular,  simply-connected $4$-dimensional solvable  Lie groups, classified in \cite{Barberis}.  By dimensional reasons they all are HKT and we show here that they actually are HKT-Einstein. Except for the abelian Lie algebra we have three cases to consider. 

\begin{es}\label{es:noncpt1}
Here we take into account the Lie algebra $ \mathfrak{aff}(\C) $ of the affine motion group of $ \C $. The structure equations are:
\[
d\zeta^1=\frac{\sqrt{-1}}{2}( \zeta^{\bar 12}-\zeta^{1\bar 2})\,, \quad d\zeta^2=\frac{\sqrt{-1}}{2}(\zeta^{1\bar 1}-\zeta^{2\bar 2})\,.
\]
The  metric $ \Omega=\zeta^{12} $ satisfies $ \alpha=-\sqrt{-1}\zeta^2 $ and thus $ \partial_J \alpha=0 $, so that the metric $ \Omega $ is HKT-Einstein with vanishing Einstein constant.
\end{es}

\begin{es}
Consider the $ 4 $-dimensional solvable Lie algebra with structure equations:
\[
d\zeta^1=\frac{1}{2}\zeta^{1\bar 1}\,, \quad d\zeta^2=-\frac{1}{2}(\zeta^{12}+ \zeta^{\bar 12})\,.
\]
The  metric $ \Omega=\zeta^{12} $ satisfies $ \alpha=-\zeta^1 $ and so $ \partial_J \alpha=-\frac{1}{2} \Omega $, showing that $ \Omega $ induces 
an invariant HKT-Einstein metric with negative Einstein constant.
\end{es}

\begin{es}
We conclude with the solvable Lie algebra with structure equations:
\[
d\zeta^1=\frac{1}{2}\zeta^{1\bar 1}-\zeta^{2\bar 2}\,, \quad d\zeta^2=-\frac{1}{4}(\zeta^{12}+ \zeta^{\bar 12})\,.
\]
The metric $ \Omega=\zeta^{12} $ satisfies $ \alpha=-\frac{3}{4}\zeta^1 $
and so $ \partial_J \alpha =-\frac{3}{16} \Omega $, which shows that $ \Omega $ is HKT-Einstein with negative Einstein constant.
\end{es}


\begin{thebibliography}{[99]}


\bibitem{Alekseevsky-Marchiafava}
{\it D. V. Alekseevsky, S. Marchiafava}, Quaternionic structures on a manifold and subordinated structures. {  Ann. Mat. Pura Appl. (4)}, {\bf 171} (1996), 205--273.

\bibitem{Alesker-Verbitsky (2006)}
{\it S. Alesker, M. Verbitsky}, Plurisubharmonic functions on hypercomplex manifolds and HKT-geometry. {  J. Glob. Anal.}, {\bf 16} (2006), 375--399.

\bibitem{Alesker-Verbitsky (2010)}
{\it S. Alesker, M. Verbitsky}, Quaternionic Monge-Amp\`{e}re equations and Calabi problem for HKT-manifolds. {  Israel J. Math.}, {\bf 176} (2010), 109--138.


\bibitem{Alexandrov-Ivanov}
{\it B. Alexandrov, S. Ivanov}, Vanishing theorems on Hermitian manifolds. {  Differential Geom. Appl.}, {\bf 14} (2001), no. 3, 251--265.


\bibitem{AT}
{\it A. Andrada, A. Tolcachier}, On the Canonical Bundle of Complex Solvmanifolds and Applications to Hypercomplex Geometry. Transformation Groups (2024).

\bibitem{ACS}
{\it D. Angella, S. Calamai, C. Spotti}, On the Chern-Yamabe problem. {  Math. Res. Lett.}, {\bf 24} (2017),  645--677.


\bibitem{AN}
{\it R. Arroyo, M. Nicolini}, SKT structures on nilmanifolds. {  Math. Z.}, {\bf 302} (2022), no.2, 1307--1320.


\bibitem{Barbarone}
{\it G. Barbaro}, On the curvature of the Bismut connection: Bismut Yamabe problem and Calabi-Yau with torsion metrics. {  J. Geom. Anal.}, {\bf 33} (2023), 153. 

\bibitem{Barberis}
{\it M. L. Barberis}, Hypercomplex structures on four-dimensional Lie groups. {  Proc. Am. Math. Soc.}, {\bf 125} (1997), 1043--1054.

\bibitem{BDV}
{\it M. L. Barberis, I. Dotti, M. Verbitsky}, Canonical bundles of complex nilmanifolds, with applications to hypercomplex geometry. {  Math. Res. Lett.}, {\bf 16} (2009), no. 2, 331--347.

\bibitem{Barberis-Fino}
{\it M. L. Barberis, A. Fino}, New HKT manifolds arising from quaternionic representations. {  Math. Z.}, {\bf 267} (2011), no. 3-4, 717--735.

\bibitem{Bedulli-Gentili-Vezzoni2}
{\it L. Bedulli, G. Gentili, L. Vezzoni}, The parabolic quaternionic Calabi-Yau equation on hyperk\"ahler manifolds. { Rev. Mat. Iberoam.}, {\bf 40} (2024), no. 6, pp. 2291--2310.


\bibitem{Boyer}
{\it C. Boyer}, A note on hyper-Hermitian four-manifolds. {  Proc. Amer. Math. Soc.}, {\bf 102} (1988), no. 1, 157--164.

\bibitem{BFG}
{\it B. Brienza, A. Fino,  G. Grantcharov}, 
A mapping tori construction of strong HKT and generalized hyperkähler manifolds, 
Ornea, Liviu (ed.), Real and complex geometry. In honour of Paul Gauduchon. Cham: Springer. 93-108 (2025).

\bibitem{BFGV}
{\it B. Brienza, A. Fino, G. Grantcharov, M. Verbitsky}, On the structure of compact strong HKT manifolds. Preprint 2025, \url{https://arxiv.org/abs/2505.06058}.

\bibitem{BFGV2}
{\it B. Brienza, U. Fowdar, G. Gentili, L.Vezzoni}, The holonomy of the Obata connection on Joyce hypercomplex manifolds, Preprint 2025, \url{https://arxiv.org/abs/2509.07722}

\bibitem{CZ}
{\it K. Cao, F. Zheng}, Fino-Vezzoni conjecture on Lie algebras with abelian ideals of codimension two. { Math. Z.}, {\bf 307} (2024), no.2, Paper No. 31, 23 pp..

\bibitem{Ch}
{\it I. Chiose}, Obstructions to the existence of K\"ahler structures on compact complex manifolds.  { Proc. Amer. Math. Soc.},  {\bf 142} (10) (2014),  3561--3568.


\bibitem{Dinew-Sroka}
{\it S. Dinew, M. Sroka}, On the Alesker-Verbitsky conjecture on hyperK\"ahler manifolds. {  Geom. Funct. Anal.}, {\bf 33} (2023), no. 4, 875--911.

\bibitem{DF1}
{\it I. Dotti and A. Fino}, Abelian hypercomplex 8-dimensional nilmanifolds. {  Ann. Glob. Anal. and Geom.}, {\bf 18} (2000), 47--59.

\bibitem{DF2}
{\it I. Dotti and A. Fino}, Hyperk\"ahler torsion structures invariant by nilpotent Lie groups. {  Classical Quantum
Gravity}, {\bf 19} (2002), 551--562.


\bibitem{Fino-Grantcharov}
{\it A. Fino, G. Grantcharov},
Properties of manifolds with skew-symmetric torsion and special holonomy. {  Adv. Math.}, {\bf 189} (2004), 439--450.

\bibitem{Fino-Grantcharov-Verbitsky}
{\it A. Fino, G. Grantcharov, M. Verbitsky}, Special Hermitian structures on suspensions. { Math. Ann.} {\bf 392} (2025), 1099--1118.

\bibitem{FGV}
{\it A. Fino, G. Grantcharov, L. Vezzoni}, Astheno-K\"ahler and balanced structures on fibrations. {  Int. Math. Res. Not., IMRN} {\bf 2019}, no. 22, 7093--7117.

\bibitem{FP}
{\it A. Fino, F. Paradiso}, Balanced Hermitian structures on almost abelian Lie algebras. {  J. Pure Applied Algebra}, {\bf 227} (2023), no. 2, Paper No. 107186. 

\bibitem{FV}
{\it A. Fino, L. Vezzoni}, Special Hermitian metrics on compact solvmanifolds. {Journal of Geometry and Physics}, {\bf 91} (2015), 40--53.

\bibitem{FV2}
{\it A. Fino, L. Vezzoni}, On the existence of balanced and SKT metrics on nilmanifolds.  { Proc. Amer. Math. Soc.}, {\bf 144} (6) (2016),  2455--2459.

\bibitem{FS}
{\it M. Freibert, A. Swann}, Compatibility of Balanced and SKT Metrics on Two-Step Solvable Lie Groups. {  Transformation Groups} {\bf 30}  (2025), 235--265.



\bibitem{bolo}
{\it E. Fusi}, The prescribed Chern scalar curvature problem. {  J. Geom. Anal.}, {\bf 32} (2022), 187.

\bibitem{Futaki}
{\it A. Futaki}, An obstruction to the existence of Einstein-K\"ahler metrics. {  Invent. Math.}, {\bf 73} (1983), 437--443.

\bibitem{cPaul}
{\it P. Gauduchon}, Le th\'{e}or\`{e}me de l'excentricit\'{e} nulle. {C. R. Acad. Sci. Paris S\'{e}r. A-B}, {\bf 285} (1977), no. 5, A387--A390.

\bibitem{Gau2}
{\it P. Gauduchon}, La 1-forme de torsion d'une vari\'{e}t\'{e} hermitienne compacte. {Math. Ann.}, {\bf 267} (1984), no. 4, 495--518. 




\bibitem{Gilbarg-Trudinger}
{\it D. Gilbarg, N. S. Trudinger}, Elliptic partial differential equations of second order, second ed., Grundlehren der Mathematischen Wissenschaften [Fundamental Principles of Mathematical Sciences], vol. {\bf 224}, Springer-Verlag, Berlin, 1983.
	
\bibitem{GP}
{\it F. Giusti and F. Podest\`{a}}, Real semisimple Lie groups and balanced metrics. {  Rev. Mat. Iberoam.}, {\bf 39} (2023), 711--729.

\bibitem{GLV}
{\it G. Grantcharov, M. Lejmi, M. Verbitsky}, Existence of HKT metrics on hypercomplex manifolds of real dimension 8. {  Adv. Math.}, {\bf 320} (2017), 1135--1157.


\bibitem{Grantcharov-Poon (2000)}
{\it G. Grantcharov, Y. S. Poon}, Geometry of hyper-K\"{a}hler connections with torsion. {  Comm. Math. Phys.}, {\bf 213} (2000), no. 1, 19--37.

\bibitem{Grover-Gutowski-Herdeiro-Sabra (2009)}
{\it J. Grover, J. B. Gutowski, C. A. R. Herdeiro, W. Sabra}, HKT geometry and de Sitter supergravity. {  Nuclear Phys. B}, {\bf 809} (2009), no. 3, 406--425. 

\bibitem{Gutowski-Papadopoulos}
{\it J. Gutowski, G. Papadopoulos}, The dynamics of very special black holes. {  Phys. Lett. B}, {\bf 472} (2000), no. 1-2, 45--53.

\bibitem{Gutowski-Sabra (2011)}
{\it J. B. Gutowski, W. Sabra}, HKT geometry and fake five-dimensional supergravity. {  Classical Quantum Gravity}, {\bf 28} (2011), no. 17, 175023, 11 pp.


\bibitem{Howe-Papadopoulos (2000)}
{\it P. S. Howe, G. Papadopoulos}, Twistor spaces for hyper-K\"{a}hler manifolds with torsion. {  Phys. Lett. B}, {\bf 379} (1996), 80--86.

\bibitem{IMV2}
{\it S. Ianu\c{s}, R. Mazzocco, G. E. V\^{\i}lcu},  Harmonic maps between quaternionic K\"ahler manifolds.  {Journal of Nonlinear Mathematical Physics}, {\bf 15} (2008), no.1, 1--8,

\bibitem{IMV}
{\it S. Ianu\c{s}, R. Mazzocco, G. E. V\^{\i}lcu},  Riemannian Submersions from Quaternionic Manifolds.  {Acta Appl. Math.}, {\bf 104} (2008),  83--89.

\bibitem{IM}
{\it S. Ivanov, I. Minchev}, Quaternionic K\"ahler and hyperK\"ahler manifolds with torsion and twistor spaces. { J. Reine Angew. Math.}, {\bf 567} (2004), 215--233.

\bibitem{Ivanov-Papadopoulos}
{\it S. Ivanov, G. Papadopoulos}, Vanishing theorems and string backgrounds. {  Classical Quantum Gravity}, {\bf 18} (2001), no. 6, 1089--1110.

\bibitem{Ivanov-Petkov}
{\it S. Ivanov, A. Petkov}, HKT manifolds with holonomy $\mathrm{SL}(n,\H)$. {  Int. Math. Res. Not., IMRN}, {\bf 16} (2012), 3779--3799. 

\bibitem{Joyce}
{\it D. Joyce}, Compact hypercomplex and quaternionic manifolds. {  J. Differential Geom.}, {\bf 35} (1992), 743--761. 

\bibitem{Kato}
{\it M. Kato}, Compact differentiable $ 4 $-folds with quaternionic structures. {  Math. Ann.}, {\bf 248} (1980), no. 1, 79--96. Erratum in {  Math. Ann.}, {\bf 283} (1989), no. 2, 352.



\bibitem{Lejmi-Weber2}
{\it M. Lejmi, P. Weber}, Cohomologies on hypercomplex manifolds. {  Complex and symplectic geometry}, 107--121, Springer INdAM Ser., {\bf 21}, {  Springer, Cham}, 2017.

\bibitem{Lejmi-Weber}
{\it M. Lejmi, P. Weber}, Quaternionic Bott-Chern Cohomology and existence of HKT metrics. {  Q. J. Math.}, {\bf 68} (2017), no. 3, 705--728.

\bibitem{Lejmi-Tardini}
{\it M. Lejmi, N. Tardini}, On the Invariant and Anti-Invariant Cohomologies of Hypercomplex Manifolds. Transformation Groups, {\bf 30} (2025), 1809--1833. 


\bibitem{Matsushima}
{\it Y. Matsushima}, Sur la structure du groupe d'hom\'{e}omorphismes analytiques d'une certaine vari\'{e}t\'{e} k\"ahl\'{e}rienne. {  Nagoya Math. J.}, {\bf 11} (1957), 145--150.

\bibitem{Mich}
{\it M.-L. Michelsohn}, On the existence of special metrics in complex geometry. {  Acta Math.}, {\bf 149} (1982), no. 3-4, 261--295.

\bibitem{Michelson-Strominger}
{\it J. Michelson, A. Strominger}, Superconformal multi-black hole quantum mechanics. {  J. High Energy Phys.}, (1999), no. 9, Paper 5, 16 pp.


\bibitem{Obata (1956)}
{\it M. Obata}. Affine connections on manifolds with almost complex, quaternionic or Hermitian structures. {  Japan. J. Math.}, {\bf 26} (1956), 43--79.

\bibitem{Opfermann-Papadopoulos}
{\it A. Opfermann, G. Papadopoulos}, Homogeneous HKT and QKT manifolds. Preprint 1998, \url{https://arxiv.org/abs/math-ph/9807026}.

\bibitem{OPS}
{\it L. Ornea, Y. S. Poon, A. Swann}, Potential 1-forms for hyper-Kähler structures with torsion. { Classical Quantum Gravity}, {\bf 20} (2003), no. 9, 1845--1856.


\bibitem{Pop09}
{\it D. Popovici}, Deformation limits of projective manifolds: Hodge numbers and strongly Gauduchon metrics. {  Invent. Math.}, {\bf 194} (2013), 515--534.
	
\bibitem{Samelson (1953)}
{\it H. Samelson}, A class of complex-analytic manifolds. {  Portugal. Math.}, {\bf 12} (1953), 129--132.

\bibitem{Sch}
{\it M. Schweitzer},
Autour de la cohomologie de Bott-Chern. 
Preprint 2007, \url{https://arxiv.org/abs/0709.3528}.

\bibitem{Siu}
{\it Y.-T. Siu}, Extension of twisted pluricanonical sections with plurisubharmonic weight and invariance of semipositively twisted plurigenera for manifolds not necessarily of general type. In Complex geometry (G\"ottingen, 2000), {  Springer-Verlag, Berlin}, 2002, 223--277.

\bibitem{SSTV}
{\it Ph. Spindel, A. Sevrin, W. Troost, W., A. Van Proeyen}, Extended supersymmetric $ \sigma $-models on group manifolds. I. The complex structures. {  Nuclear Phys. B}, {\bf 308} (1988), no. 2-3, 662--698.

\bibitem{Sroka22}
{\it M. Sroka}, Sharp uniform bound for the quaternionic Monge-Amp\`{e}re equation on hyperhermitian manifolds. Calc. Var. {\bf 63}, 102 (2024).

\bibitem{Swann}
{\it A. Swann}, Twisting Hermitian and hypercomplex geometries. 
{  Duke Math. J.}, {\bf 155} (2010), no. 2, 403--431. 



\bibitem{Verbitsky (2002)}
{\it M. Verbitsky}, HyperK\"{a}hler manifolds with torsion, supersymmetry and Hodge theory, {  Asian J. Math.}, {\bf 6}(4) (2002), 679--712.


\bibitem{Verbitsky (2007)}
{\it M. Verbitsky},	Hypercomplex manifolds with trivial canonical bundle and their holonomy. {  Moscow Seminar on Mathematical Physics}. II, 203--211,	{   Amer. Math. Soc. Transl. Ser. 2}, {\bf 221}, {  Adv. Math. Sci.}, {\bf 60}, {  Amer. Math. Soc., Providence, RI}, 2007.

\bibitem{Verbitsky (2009)}
{\it M. Verbitsky}, Balanced HKT metrics and strong HKT metrics on hypercomplex manifolds. {  Math. Res. Lett.}, {\bf 16} (2009), no. 4, 735--752.

\bibitem{Verbitsky (2010)}
{\it M. Verbitsky}, Positive forms on hyperk\"ahler manifolds. {  Osaka J. Math.}, {\bf 47} (2010), no. 2, 353--384. 

\bibitem{Vezzoni}
{\it L. Vezzoni}, A note on canonical Ricci forms on 2-step nilmanifolds. { Proc. Amer. Math. Soc.}, {\bf 141} (2013), no. 1, 325--333.

\bibitem{yu}
{\it W. Yu}, Prescribed Chern scalar curvatures on compact Hermitian manifolds with negative Gauduchon degree. {J. Funct. Anal.}, {\bf 285} (2023), no. 2, Paper No. 109948, 27 pp.


\end{thebibliography}
\end{document}